\documentclass[10pt,pdftex]{amsart}

\usepackage{lmodern}
\usepackage[bookmarks, 
colorlinks=false, backref=false,plainpages=false]{hyperref}
\usepackage[T1]{fontenc}
\usepackage[latin1]{inputenc}
\usepackage[english]{babel}
\usepackage{graphicx}
\usepackage{url}
\usepackage{graphics}
\usepackage{epsfig}
\usepackage{amsmath,amssymb,amsthm}

\renewcommand{\theequation}{\thesection.\arabic{equation}}

\newtheorem{thm}{Theorem}[section]

\newtheorem{lem}[thm]{Lemma}
\newtheorem{prop}[thm]{Proposition}
\newtheorem{rem}[thm]{Remark}

\usepackage[ruled,linesnumbered]{algorithm2e}

\begin{document}
\newcommand{\BX}{{\bf X}}
\newcommand{\cv}{{\cal V}}
\newcommand{\cW}{{\cal W}}
\newcommand{\co}{{\cal O}}

\renewcommand{\theequation}{\thesection.\arabic{equation}}
\def\@eqnnum{{\reset@font\rm (\theequation)}}

\def\abstract{
\advance \rightskip by 10mm
\advance \leftskip by 10mm
\vspace{-0.8em}
\noindent
\small{\bf Abstract.}
}
\def\endabstract{\par\normalsize\rm}

\def\Xint#1{\mathchoice
{\XXint\displaystyle\textstyle{#1}}%
{\XXint\textstyle\scriptstyle{#1}}%
{\XXint\scriptstyle\scriptscriptstyle{#1}}%
{\XXint\scriptscriptstyle\scriptscriptstyle{#1}}%
\!\int}
\def\XXint#1#2#3{{\setbox0=\hbox{$#1{#2#3}{\int}$}
\vcenter{\hbox{$#2#3$}}\kern-.5\wd0}}
\def\ddashint{\Xint=}
\def\dashint{\Xint-}

\def\a{\alpha}
\def\b{\beta}
\def\d{\delta}\def\D{\Delta}
\def\e{\epsilon}
\def\g{\gamma}\def\G{\Gamma}
\def\k{\kappa}
\def\lam{\lambda}\def\Lam{\Lambda}
\renewcommand\o{\omega}\renewcommand\O{\Omega}
\def\s{\sigma}\def\S{\Sigma}
\renewcommand\t{\theta}\def\vt{\vartheta}
\newcommand{\vphi}{\varphi}
\def\z{\zeta}

\newcommand{\tsigma}{\tilde{\s}}
\newcommand{\tbsigma}{\tilde{\bsigma}}
\def\te{\tilde{\e}}
\def\tu{\tilde{u}}

\newcommand{\bchi}{\mbox{\boldmath$\chi$}}
\newcommand{\bdelta}{\mbox{\boldmath$\delta$}}
\newcommand{\bepsilon}{\mbox{\boldmath$\epsilon$}}
\newcommand{\bfeta}{\mbox{\boldmath$\eta$}}
\newcommand{\bgamma}{\mbox{\boldmath$\gamma$}}
\newcommand{\bomega}{\mbox{\boldmath$\omega$}}
\newcommand{\bvphi}{\mbox{\boldmath$\varphi$}}
\newcommand{\bphi}{\mbox{\boldmath$\phi$}}
\newcommand{\bPhi}{\mbox{\boldmath$\Phi$}}
\newcommand{\bpsi}{\mbox{\boldmath$\psi$}}
\newcommand{\bPsi}{\mbox{\boldmath$\Psi$}}
\newcommand{\bsigma}{\mbox{\boldmath$\sigma$}}
\newcommand{\btau}{\mbox{\boldmath$\tau$}}
\newcommand{\bxi}{\mbox{\boldmath$\xi$}}
\newcommand{\brho}{\mbox{\boldmath$\rho$}}

\newcommand{\bbeta}{\mbox{\boldmath$\beta$}}
\newcommand{\bzeta}{\mbox{\boldmath$\zeta$}}

\def\bk{\boldsymbol{\kappa}}
\def\bmu{\boldsymbol\mu}
\def\bxi{\boldsymbol{\xi}}
\def\bz{\boldsymbol{\zeta}}

\def\ba{{\bf a}}
\def\bb{{\bf b}}
\def\bc{{\bf c}}
\def\be{{\bf e}}
\def\bff{{\bf f}}
\def\bg{{\bf g}}
\def\bn{{\bf n}}
\def\bp{{\bf p}}
\def\bq{{\bf q}}
\def\bs{{\bf s}}
\def\bt{{\bf t}}
\def\bu{{\bf u}}
\def\bv{{\bf v}}
\def\bw{{\bf w}}
\def\bx{{\bf x}}
\def\by{{\bf y}}
\def\bzz{{\bf z}}

\def\bD{{\bf D}}
\def\bE{{\bf E}}
\def\bF{{\bf F}}
\def\bH{{\bf H}}
\def\bJ{{\bf J}}
\def\bV{{\bf V}}
\def\bU{{\bf U}}
\def\bW{{\bf W}}
\def\bX{{\bf X}}
\def\bY{{\bf Y}}

\def\cA{{\cal A}}
\def\cC{{\cal C}}
\def\cD{{\cal D}}
\def\cE{{\cal E}}
\def\cF{{\cal F}}
\def\cG{{\cal G}}
\def\cI{{\cal I}}
\def\cJ{{\cal J}}
\def\cK{{\cal K}}
\def\cL{{\cal L}}
\def\cO{{\cal O}}
\def\cP{{\cal P}}
\def\cQ{{\cal Q}}
\def\cR{{\cal R}}
\def\cS{{\cal \Sigma}}
\def\cT{{\cal T}}
\def\cU{{\cal U}}
\def\cV{{\cal V}}

\def\scT{{_\cT}}
\def\sD{{_D}}
\def\sE{{_E}}
\def\sF{{_F}}
\def\sFz{{_{F_z}}}
\def\sK{{_K}}
\def\sI{{_I}}
\def\sb{{_b}}
\def\sN{{_N}}

\def\curl{{{\bf curl} \ }}
\def\rot{{\mbox{rot}\ }}
\def\BPI{{\bf \Pi}}

\def\cth{\cT_h}
\def\ctH{\cT_H}

\def\tJ{\tilde{\J}}

\def\hK{\widehat{K}}
\def\hx{\widehat{x}}
\def\hy{\widehat{y}}
\def\bhv{\widehat{\bv}}

\def\l{\ell}
\def\bl{\boldsymbol{\ell}}
\def\col{\colon}
\def\f12{\frac12}
\def\dfrac{\displaystyle\frac}
\def\dint{\displaystyle\int}
\def\nab{\nabla}
\def\p{\partial}
\def\sm{\setminus}
\def\dsum{\displaystyle\sum}
\newcommand{\pp}[2]{\frac{\partial {#1}}{\partial {#2}}}
\def\bzero{{\bf 0}}

\def\divv{\nab\cdot}
\def\divx{\nab_x\cdot}
\def\divtx{\nab_{t,x}\cdot}
\def\nabx{\nab_x}

\newcommand{\grad}{\nabla}
\newcommand{\curlt}{{\nabla \times}}
\newcommand{\gperp}{\nabla^{\perp}}
\newcommand{\gradt}{\nabla\cdot}

\def\forallqq{\quad\forall\,}
\def\aph{A^{1/2}}
\def\amh{A^{-1/2}}

\def\osc{{\rm osc \, }}

\def\Im{{\rm Im}}
\newcommand{\tr}{{\rm tr}}
\def\divvr{{\rm div}}
\def\curllr{{\rm curl}}
\def\curll{{\rm curl}}
\def\curl{{\bf curl}}
\newcommand{\bgrad}{{\bf grad}}
\newcommand\diam{\mathrm{diam\,}}
\renewcommand\Im{\mathrm{Im\,}}
\def\Span{\mbox{Span}}
\def\supp{\mbox{supp\,}}
\newcommand{\trace}{{\rm trace}}

\newcommand{\tri}{|\!|\!|}
\newcommand{\ljump}{\lbrack\!\lbrack}
\newcommand{\rjump}{\rbrack\!\rbrack}
\newcommand{\bdm}{\begin{displaymath}}
\newcommand{\edm}{\end{displaymath}}
\newcommand{\beq}{\begin{equation}}
\newcommand{\eeq}{\end{equation}}
\newcommand{\beqa}{\begin{eqnarray}}
\newcommand{\eeqa}{\end{eqnarray}}
\newcommand{\beqas}{\begin{eqnarray*}}
\newcommand{\eeqas}{\end{eqnarray*}}
\newcommand{\ul}{\underline}
\newcommand{\wh}{\widehat}
\newcommand{\la}{\langle}
\newcommand{\ra}{\rangle}

\newcommand{\Lt}{L^2(\Omega)}
\newcommand{\Lts}{L^2(\Omega)^2}
\newcommand{\Ltc}{L^2(\Omega)^3}
\newcommand{\Ho}{H^1(\Omega)}
\newcommand{\Hoh}{H^1(\wh{\Omega})}
\newcommand{\Hoi}{H^1(\Omega_i)}
\newcommand{\Hos}{H^1(\Omega)^2}
\newcommand{\Hoc}{H^1(\Omega)^3}
\newcommand{\Hoch}{H^1(\wh{\Omega})^3}
\newcommand{\Hoci}{H^1(\Omega_i)^3}
\newcommand{\Hoz}{H^1_0(\Omega)}
\newcommand{\Ht}{H^2(\Omega)}
\newcommand{\Hti}{H^2(\Omega_i)}
\newcommand{\Hts}{H^2(\Omega)^2}
\newcommand{\Htc}{H^2(\Omega)^3}
\newcommand{\Htz}{H^0(\Omega)}
\newcommand{\Hh}{H^{1/2}(\Gamma)}
\newcommand{\Hhi}{H^{1/2}(\Gamma_i)}
\newcommand{\Hmh}{H^{-1/2}(\Gamma)}
\newcommand{\Hdiv}{H(\divvr;\,\Omega)}
\newcommand{\Hdivh}{H(\divv;\,\wh \Omega)}
\newcommand{\hcurl}{H(\curl\,A;\,\Omega)}
\newcommand{\Hcurl}{H(\curll\,A;\,\Omega)}
\newcommand{\Hcrl}{H(\curll\,;\,\Omega)}
\newcommand{\hcrl}{H(\curl\,;\,\Omega)}
\newcommand{\Hcrlh}{H(\curll\,;\,\wh\Omega)}
\newcommand{\hcrlh}{H(\curl\,;\,\wh\Omega)}
\newcommand{\Wdiv}{\BW_0(\mbox{\divv}\,;\,\Omega)}
\newcommand{\Wcurl}{\BW_0(\mbox{\curl}\,A;\,\Omega)}
\newcommand{\WcrossV}{\BW \times V}

\def\grad{{\nabla}}

\def\calS{{\cal S}}
\def\calT{{\cal T}}
\def\cA{{\mathcal A}}
\def\cB{{\cal B}}
\def\cD{{\mathcal{D}}}

\def\cH{{\cal H}}
\def\cU{{\mathcal U}}
\def\ba{{\mathbf{a}}}

\def\beps{{\mathbf{\epsilon}}}

\def\cM{{\mathcal{M}}}
\def\cN{{\mathcal{N}}}
\def\cT{{\mathcal{T}}}
\def\cE{{\mathcal{E}}}
\def\cP{{\mathcal{P}}}
\def\cF{{\mathcal{F}}}

\def\cB{{\mathcal{B}}}
\def\cG{{\mathcal{G}}}

\def\cL{{\mathcal{L}}}
\def\cJ{{\mathcal{J}}}

\def\cR{{\mathcal{R}}}

\def\cV{{\mathcal{V}}}
\def\cW{{\mathcal{W}}}
\def\cZ{{\mathcal{Z}}}

\def\bbT{{\mathbb{T}}}
\def\bbV{{\mathbb{V}}}
\def\bbW{{\mathbb{W}}}
\def\bbX{{\mathbb{X}}}

\newcommand{\lJump}{[\![}
\newcommand{\rJump}{]\!]}
\newcommand{\jump}[1]{[\![ #1]\!]}

\newcommand{\sd}{\bsigma^{\Delta}}
\newcommand{\rd}{\brho^{\Delta}}

\newcommand{\eps}{\epsilon}
\newcommand{\R}{\rm I\kern-.19emR}

\newcommand{\LS}{{\mathtt{LS}}}
\newcommand{\NLS}{{\mathtt{NLS}}}
\newcommand{\NJ}{{\mathtt{NJ}}}
\newcommand{\JJ}{{\mathtt{J}}}

\newcommand{\blt}{{\bullet}}

\newcommand{\norm}[1]{\left\lVert#1\right\rVert}

\newcommand{\refine}{\mbox{refine}}

\title [Non-intrusive Least-Squares Functional A Posteriori Error Estimator]{
Non-intrusive Least-Squares Functional A Posteriori Error Estimator:  Linear and Nonlinear Problems with Plain Convergence}
\author[Z. Li and S. Zhang]{Ziyan Li and Shun Zhang}
\address{Department of Mathematics, City University of Hong Kong, Kowloon Tong, Hong Kong, China}
\email{ziyali3-c@my.cityu.edu.hk, shun.zhang@cityu.edu.hk}
\thanks{This work was supported in part by
Research Grants Council of the Hong Kong SAR, China, under the GRF Grant Project No. CityU 11316222}
\date{\today}


\maketitle
\begin{abstract}
The a posteriori error estimator using the least-squares functional can be used for adaptive mesh refinement and error control even if the numerical approximations are not obtained from the corresponding least-squares method. This suggests the development of a versatile non-intrusive a posteriori error estimator. 
In this paper, we present a systematic approach for applying the least-squares functional error estimator to linear and nonlinear problems that are not solved by the least-squares finite element methods. 
For the case of an elliptic PDE solved by the standard conforming finite element method, we minimize the least-squares functional with conforming approximation inserted to recover the other physical meaningful variable. By combining the numerical approximation from the original method with the auxiliary recovery approximation, we construct the least-squares functional a posteriori error estimator. Furthermore, we introduce a new interpretation that views the non-intrusive least-squares functional error estimator as an estimator for the combined solve-recover process. This simplifies the reliability and efficiency analysis. We extend the idea to a model nonlinear problem. Plain convergence results are proved for adaptive algorithms of the general second order elliptic equation and a model nonlinear problem with the non-intrusive least-squares functional a posteriori error estimators.

\end{abstract}


\section{Introduction}\label{intro}
\setcounter{equation}{0}
The least-squares finite element method (LSFEM) \cite{fosls,CLMM:94,CMM:97,BG:98,Jiang:98,BG:09,CFZ:15,Zhang:23} designs numerical methods based on minimizations of least-squares energy functionals with first-order system reformulations. Compared to the standard variational formulation and the related finite element method, one of the main advantages of the LSFEM is that the least-squares functional is a good a posteriori error indicator/estimator for both the mesh refinement and the error control. Earlier examples of adaptive LSFEMs can be found in \cite{JC:87,BMM:97}. For a series of problems, the built-in least-squares functional error estimators have been studied, for example, \cite{DMMO:04,Sta:05,CW:09,MS:11,LZ:18,LZ:19,QZ:20,Fuh:20,Bringmann:23}. 

To illustrate the idea, consider the following first-order system,  which may be linear or nonlinear, 
$$
B(\phi,\psi)^t = F
$$ in the sense of $L^2$. We assume that homogeneous boundary conditions and assume that $\phi \in W$ and $\psi\in V$ with $W$ and $V$ being the $L^2$-based Hilbert spaces. Define the least-squares energy functional as:
$
E(\alpha,\beta;F) = \norm{B (\alpha,\beta)^t - F}_0^2.
$
Then the least-squares minimization problem is:
$$
(\phi, \psi) = \mbox{arg min}_{\alpha\in W, \beta \in V} E(\alpha,\beta;F).$$
Let $W_h\subset W$ and $V_h\subset V$ be two corresponding finite element spaces, then the the least-squares finite element minimization problem is: 
$
(\phi_h, \psi_h) = \mbox{arg min}_{\alpha\in W_h, \beta \in V_h} E(\alpha,\beta;F)$.

There are two central ideas in the construction of a posteriori error estimators. The first one is solving the problem twice with different approaches and the second one is the evaluation or estimation of the residual in specific norms. In the adaptive LSFEM framework, we use the least-squares functional error estimator given by:
$$
\eta(\phi_h,\psi_h) : = \norm{B (\phi_h,\psi_h)^t - F}_0 = \sqrt{E(\phi_h,\psi_h;F)}
$$
This error estimator naturally combines two concepts: the first-order system LSFEM, which solves the equations in physically meaningful variables (e.g., $\phi$ and $\psi$), and the least-squares functional, which evaluates the residual $\mathbf{B} (\phi_h,\psi_h)^t - F$ in $L^2$-norm.

For the well-studied problems, we often have the so-called norm-equivalence, where the least-squares functional estimator is equivalent to the error measured in some standard norms.  Even in cases where the norm-equivalence is not available, the least-squares functional estimator can still be used since the energy functional induces an artificial least-squares energy norm, see \cite{LZ:18,LZ:19,QZ:20} for applications in transport equations and elliptic equations in non-divergence form. 
The adaptive least-squares method, as a "brute-force" approach, has advantages for its directness and simplicity. Recent examples of using least-squares methods as a "brute-force" method for solving PDEs numerically includs the physics-informed neural networks \cite{PINN}.  

One distinctive feature of the least-squares functional estimator, which sets it apart from many other a posteriori error estimators, is that its reliability and efficiency bounds do not require the discrete approximations to be exact solutions of the underlying LSFEM problem. We can use the estimator
$
\eta(\alpha_h,\beta_h) : = \|B (\alpha_h,\beta_h)^t - F\|_0
$
with $\alpha_h \in W_h$ and $\beta_h\in V_h$ as  inexact approximations of $(\phi,\psi)$. In contrast, most efficiency bounds of other a posteriori error estimators rely on Verfurth's bubble function trick \cite{Ver:13}, which is based on the error equation. Consequently, these estimators typically require the numerical approximation to be the solution of the corresponding finite element discrete problem. Thus, designing and analyzing a posteriori error estimators for specific problems often necessitate individual treatments.

While the LSFEM offers several advantages as a "brute-force" and adaptive method, it is important to acknowledge that many real-world problems already have well-established numerical methods that are often not based on least-squares techniques, and may not even use finite elements or variational formulations. These methods may incorporate specific tricks tailored to the problem at hand. In this paper, we aim to address the following questions: Can we use the least-squares functional a posteriori error estimator even when the underlying numerical method is not based on least-squares principles? In other words, can we use the least-squares functional error estimator in a non-intrusive manner? Furthermore, can we ensure that the adaptive algorithm driven by the proposed non-intrusive least-squares functional a posteriori error estimator converges? 
These questions are crucial as they explore the applicability and effectiveness of the least-squares functional estimator beyond its original context. By investigating these aspects, we aim to provide insights into the potential use of the least-squares functional estimator as a versatile tool that can complement and enhance existing numerical methods for various problem types. 

In this paper, we aim to achieve the following objectives: 
(1) Develop a non-intrusive least-squares functional a posteriori error estimator for problems that are solved by standard finite element methods.  (2) Establish a framework for a priori and a posteriori error analysis by considering the solve-recover process as a two-step combined problem. (3) Prove the plain convergence of the adaptive algorithm driven by the non-intrusive least-squares functional a posteriori error estimator. 
To illustrate the effectiveness and generality of our proposed method, we apply it to two model problems: the general indefinite and non-symmetric second-order elliptic equation and a monotone nonlinear problem.

The idea of non-intrusive least-squares functional a posteriori error estimators for a problem which is not solved by a LSFEM is originally proposed in \cite{CZ:10b} for elliptic equations. In that work, for an elliptic equation solved by the conforming finite element method, the flux is recovered by minimizing the least-squares functional with the conforming finite element solution inserted. The least-squares functional a posteriori error estimator can then be applied using the conforming finite element approximation and the recovered flux approximation.
This concept can be generalized to different problems in order to develop non-intrusive least-squares functional estimators. The approach involves first solving the PDE using a preferred numerical method. For example, we may solve for an approximation $\phi_h\in W_h$ using the standard Galerkin method. If certain necessary auxiliary variables are missing for the application of the least-squares functional  estimator, we can recover them in appropriate discrete spaces by minimizing the functional with the existing finite element solution inserted. For instance, if $\psi_h$ is missing, we can recover $\psi_h$ by solving the minimization problem:
\beq \label{LS_recover}
\psi_h = \mbox{arg} \min_{\beta \in V_h} E(\phi_h,\beta;F).
\eeq
By obtaining the recovered variable $\psi_h$, we can then apply the least-squares functional a posteriori error estimator $\eta(\phi_h,\psi_h) = \|B (\phi_h,\psi_h)^t - F\|_0$ to estimate the error.

For this non-intrusive least-squares functional a posteriori error estimator, we need to discuss its reliability and efficiency. Typically, when constructing an a posteriori error estimator for a numerical method, it needs to be proven that the error measured in a certain norm is bounded both above and below by the error estimator, up to certain constants and high-order perturbations. However, establishing such bounds, particularly the efficiency bound, for the non-intrusive least-squares functional a posteriori error estimator is challenging.
Taking the similar flux-recovery estimator as an example, the efficiency bound is proved by constructing an explicit recovery and finding its relation with the known residual-type error estimator. However, this approach poses additional difficulties for the non-intrusive least-squares functional estimator, as its purpose is to be applied to less-studied problems where a residual-based error estimator may not be available. In this paper, we introduce a novel interpretation of the non-intrusive least-squares functional a posteriori error estimator. We consider the solve and least-squares recovery processes as a combined two-step problem. In the solve step, we obtain the numerical approximation of the original unknown quantity ($\phi_h$), while in the least-squares recovery step, we determine the numerical approximation of the auxiliary unknown quantity ($\psi_h$). The least-squares functional a posteriori error estimator can then be understood as an estimator for both approximations $(\phi_h,\psi_h)$. This interpretation simplifies the mathematical analysis required to establish the reliability and efficiency of the non-intrusive least-squares functional a posteriori error estimator. As long as the approximations belong to the underlying function spaces, the least-squares functional a posteriori error estimator $\eta(\phi_h,\psi_h)$ is reliable and efficient for the combined approximations. Moreover, this viewpoint opens up opportunities for further applications of this idea to less-explored problems where a residual-based error estimator may not be available. The notion that the least-squares functional a posteriori error estimator can be seen as an estimator for the solve-recover combined problem is implicitly utilized in the original paper \cite{CZ:10b}.

The concept of using the least-squares functional  estimator alone, without solving the underlying problem using LSFEM, is particularly appealing for nonlinear problems. In the case of a nonlinear problem, the least-squares minimization problem becomes non-convex, even if the original problem is convex. Consequently, the standard LSFEM can encounter issues related to the non-uniqueness of the discrete minimizer. Additionally, the least-squares approach amplifies the nonlinearity of the problem.
However, by employing a well-established numerical method to solve the primal variable and only recovering the auxiliary variable using partial least-squares \eqref{LS_recover} (which is often a linear problem), we can use the estimator without solving a nonlinear least-squares problem. In this approach, both the solve step and the recovery step are well-studied and supported by existing techniques. Consequently, we can use the simple least-squares functional estimator. The additional computational cost introduced by the recovery step is acceptable since it involves solving only a one-shot linear problem. Even in the case of a linear problem, the cost of solving the global recovery problem remains reasonable since it is comparable to the computational cost of solving the problem using a standard LSFEM approach.


Based on this combined two-step  problem framework of a priori and a posteriori analysis, we can discuss the plain convergence of the adaptive algorithm driven by the non-intrusive least-squares functional estimator. Similar to the recent work by F\"uhrer and Praetorius \cite{FP:20}, we establish that the combined two-step problem satisfies a set of conditions within the abstract framework proposed by Siebert \cite{Sie:11}. 
Convergence theories for standard adaptive finite element methods driven by residual-based error estimators with optimal rates have been established in \cite{Dof:96, MNS:02, BDD:04, CFPP:14}. The plain convergence of adaptive finite element methods has been studied in \cite{MSV:08, Sie:11}. However, since the least-squares functional a posteriori error estimator does not contain mesh-size factors in its terms, the standard arguments utilized in \cite{Dof:96, MNS:02, BDD:04, CFPP:14} to prove convergence with rates for adaptive finite element methods cannot be directly applied.
Alternative explicit residual-based error estimators specifically designed for LSFEMs have been developed in \cite{CP:15, BC:17, Bri:23}, and optimal convergence rates have been established for these estimators. However, in this paper, we do not pursue this approach as our main objective is to use the original least-squares functional estimator. Consequently, we focus solely on the plain convergence analysis. The plain convergence of standard adaptive LSFEMs can be found in \cite{CPB:17, FP:20, GS:21}.


In the remaining sections of the paper, our focus is primarily on two model problems: the general indefinite and non-symmetric second-order elliptic equation and a monotone nonlinear problem. For both of these equations, which are solved using standard conforming finite element methods, we derive the non-intrusive least-squares functional estimators. The a priori and a posteriori error analysis are developed and the plain convergence is proved in details.
The plain convergence analysis also confirms that even the least-squares functional error estimator is naturally reliable and efficient,  we still require a priori convergence in the corresponding least-squares functional equivalent norms to ensure the convergence of the adaptive algorithms. 


The concept of recovering an additional physically meaningful variable and utilizing it to construct a posteriori error estimators is not a new idea. One notable example is the ZZ estimator \cite{ZZ1:92, ZZ2:92}. The least-squares functional error estimator, based on auxiliary recovery, can also be seen as a natural extension of the duality-gap error estimator, which is built on primal-dual variational principles \cite{BS:08, CZ:12, EV:15, Zhang:20}. 
For a certain class of problems that satisfy a natural minimization principle and possess a dual problem associated with a natural maximizing principle, the duality gap between the primal and dual approximations serves as a reliable estimator for the error. This is discussed in detail in \cite{Zhang:20}. However, for problems that lack a natural energy minimization principle, we must seek an artificial energy functional, such as the least-squares energy functional, and recover the dual or auxiliary variable by minimizing this artificial energy functional. This approach leads to the development of the non-intrusive least-squares functional error estimator.


The paper is organize as follows: In Section 2,Sobolev spaces, finite element meshes, refinements, and corresponding finite element spaces are discussed. The general second-order elliptic equation and its conforming finite element approximation are presented in Section 3. In Section 4, we review the LSFEM and its built-in least-squares functional error estimator. In Section 5, we define the non-intrusive least-Squares functional error estimator for conforming FEM of elliptic equation. The alternative view on the non-intrusive least-squares functional estimator is discussed on Section 6 and the plain convergence is proved for the adaptive algorithm driven by the non-intrusive least-squares functional estimator in Section 7. In Section 8, we discuss a simple monotone nonlinear problem and its conforming finite element a priori error estimate. Non-intrusive least-squares functional error estimator for the model monotone problem is developed in Section 9 and plain convergence is proved in Section 10. Some final comments and remarks are given in Section 11.

\section{Sobolev spaces, meshes, refinement, and finite element spaces}
\subsection{Sobolev spaces}
For a domain $\omega$, we use the notation  $\|v\|_{0,p,\omega}$ for the $L^p$ norm of a function  $v\in L^p(\omega)$. When $p=2$,  the simpler notation  $\|v\|_{0,\omega}$ for the $L^2$ norm of a function  $v\in L^2(\omega)$ is used. The $H^1$-norm for a function $v\in H^1(\omega)$ is $\|v\|_{1,\omega}$, and the norm of a vector function $\btau \in H(\divvr;\omega)$ is 
$
\|\btau\|_{H(\divvr;\omega)} : =\left( \|\btau\|_{0,\omega}^2 + \|\gradt\btau\|_{0,\omega}^2\right)^{1/2}.
$ 
The space $H^1_0(\omega)$ is the subspace of $H^1(\omega)$ with zero boundary condition. Let 
$$ \bbX(\omega) = H(\divvr;\omega)\times  H^1_0(\omega),$$ and define its norm as
\begin{equation*}
         \tri (\btau,v) \tri_{\omega} : = \left( \|v\|_{1,\omega}^2 + \|\btau\|_{H(\divvr;\omega)} ^2\right)^{1/2}.
\end{equation*}
The domain $\O$ is a bounded, open, connected subset of $\mathbb{R}^d$ ($d = 2$ or $3$) with a Lipschitz continuous boundary $\p\O$. When $\omega= \Omega$, we omit the subscript $\O$ for simplicity. The following proposition is obvious, which is the Assumptions (A3) and (A4) of \cite{FP:20} and (2.3) of \cite{Sie:11}.
\begin{prop}\label{A3A4}
The norm on $\bbX$ is additive and absolutely continuous with respect to the Lebesgue measure.
\end{prop}

\subsection{Meshes and mesh-refinement}
We assume that $\cT_\bullet$ is a conforming simplicial triangulation of the domain $\O$. Define the mesh-size function $h_\bullet \in L^{\infty}(\O)$ for the mesh $\cT_\bullet$ to be $h_\bullet|_T = h_T = |T|^{1/d}$, for any $T\in\cT_\bullet$.

Let $\cM_\bullet\subset \cT_\bullet$ to be a set of marked elements. Let $\mbox{refine}(\cdot)$ be the standard newest vertex bisection refinement routine.  Refine at least all marked elements in $\cM_\bullet$, we get a new mesh $\cT_\oplus $ from $\cT_\bullet$. This process is denoted by $\cT_\oplus    = \refine(\cT_\bullet,\cM_\bullet)$. The notation $\bbT(\cT_\bullet)$ represents the collection of all meshes that can be obtained through an arbitrary but finite number of refinements of $\cT_\bullet$. Let $\cT_0$ be the initial mesh and $\bbT:=\bbT(\cT_0)$.

We make the following assumptions (which essentially coincide with (2.4) of \cite{Sie:11} and {\bf R1, R2,} and {\bf R3} of \cite{FP:20}):
\begin{description}
\item[Assumption R1]  {\bf Reduction on refined elements:} On refined elements, the mesh-size function is monotone and contractive, with a constant $0 < q_{\mbox{ref}} < 1$, i.e.,
$$
h_\oplus \leq h_\bullet \mbox{  a.e. in }\O \mbox{ and } 
h_\oplus \leq q_{\mbox{ref}} h_\bullet \quad \forall T \in  \cT_\bullet\backslash \cT_\oplus.
$$

\item[Assumption R2] {\bf Uniform shape regularity:} There exists a positive constant $\kappa$, which depends only on the initial mesh $\cT_0$, such that
$$
\mbox{diam}(T)^d \leq \kappa |T| \quad \forall T\in \cT_\bullet, \cT_\bullet \in \bbT.
$$

\item[Assumption R3] {\bf Marked elements are refined:} The following is true:
$$
\cM_\bullet \cap \refine(\cT_\bullet, \cM_\bullet) =\emptyset \quad \forall \cT_\bullet \in \bbT \mbox{  and  } \cM_\bullet \subset \cT_\bullet.
$$
\end{description}

\subsection{Finite element spaces}

For each mesh $\cT_\bullet\in\bbT$, we use the $C^0$-conforming finite element space to approximate $H^1$-functions. Let $P_n(T)$ be the space of polynomials of degree $n$ on an element $T\in \cT_\bullet$. Denote the linear $C^0$-conforming finite element space associated with the triangulation $\cT_\bullet$ by
\beq \label{C0FE}
\bbV_\blt :=\{v\in H^1_0(\O):v|_T \in P_1(T)\; \forall T\in\cT_\bullet\} \subset H^1_0(\O).
\eeq
Denote the local lowest-order Raviart-Thomas (RT) \cite{RT:77} on an element $T\in\cT_\bullet$ by $RT_{0}(T)=P_{0}(T)^d +\bx\,P_{0}(T)$. The $\Hdiv$ conforming $RT_0$ space is defined by
\beq \label{RTFE}
\bbW_\blt := \{\btau\in H(\divvr;\Omega):  \btau|_T\in RT_{0}(T)\;\;\forall\,\,T\in\cT_{\bullet}\}.
\eeq
We use the notation 
$$
\bbX_\blt :=  \bbW_\blt \times \bbV_\blt.
$$
We have the following property of the discrete space $\bbX_\blt$, which coincide with Assumptions (S1) and (S2) of \cite{FP:20} and  (3.5) of \cite{Sie:11}.

\begin{prop} \label{S1S2}
The space $\bbX_\blt$ (also $\bbW_\blt$ and $\bbV_\blt$) are conforming and finite dimensional for all $\cT_\blt\in \bbT$. The mesh-refinement ensures $\bbX_\blt$  is nested, that is $\bbX_\blt \subset \bbX_\oplus$ for all $\cT_\oplus \in \bbT(\cT_\blt)$.
\end{prop}
We discuss the approximation properties of $\bbW_\blt$ and $\bbV_\blt$. To get a priori error estimate, we will make low regularity assumptions on the regularity of the solution, while we only need approximation properties for $H^2$-functions if we just want to prove the plain convergence of adaptive algorithms. 

By Sobolev's embedding theorem, $H^{1+s}(\O)$, with $s>0$ for two dimensions and $s > 1/2$ for three dimensions, is embedded in $C^0(\O)$. Thus, we can define the nodal interpolation $I^{nodal}_\blt$ of a function $v\in H^{1+s}(\O)$ with $I^{nodal}_\blt v \in \bbV_\blt$ and $I^{nodal}_\blt v(z) = v(z)$ for a vertex $z$. 
We have the following local interpolation estimate for the linear nodal interpolation $I^{nodal}_\blt$ with local regularity $0<s_T\leq 1$ in two dimensions and $1/2<s_T\leq 1$ in there dimensions \cite{DuSc:80,CHZ:17}:
\beq \label{nodal_inter}
\|v- I^{nodal}_\blt v\|_{1,T} \leq C h_T^{s_T}\|v\|_{1+s_T,T} \quad \forall T\in \cT_\blt.
\eeq
For solutions with low regularities, the nodal interpolation is not well-defined. We can use the modified Cl\'{e}ment interpolation \cite{Clement:75,BeGi:98} or the Scott-Zhang interpolation \cite{SZ:90}. For an element $T\in\cT_\blt$, let $\Delta_T$ be the collection of elements in $\cT_\blt$ that share at least one vertex with $T$. Assume that $v\in H^1_0(\O)$ and $v|_{\Delta_T}\in H^{1+s_{\Delta_T}}(\Delta_T)$ for some $0 <  s_{\Delta_T}\leq k+1$, and let $I_{sz} v$ be the Scott-Zhang interpolation into $S_{k+1,0}$, we have 
\beq\label{sz}
\|\nabla(v-I_{sz} v)\|_{0,T} \leq C h_T^{s_{\Delta_T}}|v|_{1+s_{\Delta_T},\Delta_T}.
\eeq
Define $\cT_{\blt,s}$ to be the part of the mesh such that the local element-wise regularity $s_T$ of $H^{1+s_T}(T)$ is  big enough to ensure the nodal interpolation:
$$
\cT_{\blt,s} :=\{T\in \cT_\blt: s_T>0 \mbox{ for } d=2 \mbox{ and } s_T>1/2  \mbox{ for }  d=3 \}.
$$
Assume that $\btau\in L^r(\O)^d\cap H(\divvr;\O)$ for some $r>2$, and locally $\btau \in H^{s_T}(T)$ with the local regularity  $1/2<s_T\leq 1$. Let $I^{rt}_h$ be the canonical RT interpolation from $L^r(\O)^d\cap H_N(\divvr;\O)$ to $RT_{0,N}$. Then the following local interpolation estimates hold for local regularity $1/2<s_T\leq 1$ with the constant $C_{rt}$ being unbounded as $s_T\downarrow 1/2$ (see Chapter 16 of \cite{FE1}): 
\beq \label{RT_inter1}
\|\btau- I^{rt}_\blt\btau\|_{0,T} \leq C_{rt} h_T^{s_T}\|\btau\|_{s_T,T} \quad \forall T\in \cT_\blt.
\eeq
Due to the commutative property of the standard RT interpolation, if we further assume that $\gradt \btau|_T \in H^{t_T}(T)$, $0<t_T\leq 1$, then
\beq \label{RT_inter2}
\|\gradt (\btau- I^{rt}_\blt\btau)\|_{0,T} \leq C h_T^{t_T}|\gradt\btau|_{t_T,T} \quad \forall T\in \cT_\blt.
\eeq
To prove the plain convergence, we only need the following approximation property in the dense subspace of $\bbX$. Note that $H^2(\O)\cap H_0^1(\O)$ is dense in $H^1_0(\O)$ and $H^2(\O)^d$ is dense in $H(\divvr;\O)$. Define the interpolation operator 
$$
I_\blt: H^2(\O)^d\times (H^2(\O)\cap H_0^1(\O)) \rightarrow \bbX_\blt \mbox{ defined by }I_\blt(\btau,v):=(I^{rt}_\blt \btau, I^{nodal}_\blt v).
$$
For $\btau \in H^2(T)$, we then have $\|\btau- I^{rt}_\blt\btau\|_{0,T} \leq C h_T\|\btau\|_{1,T} \leq C h_T\|\btau\|_{2,T}$ and $\|\gradt (\btau- I^{rt}_\blt\btau)\|_{0,T} \leq C h_T\|\gradt\btau\|_{1,T} \leq C h_T\|\btau\|_{2,T}$. Thus,  the interpolation operator $I_\blt$ satisfies
\beq \label{localapp}
\tri (\btau,v) - I_\blt(\btau,v) \tri_{T} \leq C h_T (\|\btau\|_{H^2(T)}+\|v\|_{H^2(T)}) \quad \forall 
(\btau,v) \in H^2(\O)^d\times (H^2(\O)\cap H_0^1(\O)) \mbox{ and } \forall T\in\cT_\blt.
\eeq
Thus, the approximation operator $I_\blt$ satisfies the assumption (S3) ({local approximation property}) of \cite{FP:20} and (2.5c) of \cite{Sie:11} with the square of the $H^2$-norm  is obviously addictive.

\section{General second-order elliptic problem and its conforming FEM}
Consider the second order elliptic problem:
\beq \label{pde1_m1}
-\gradt(A \nabla u) +\bb\cdot \nabla  u + c u =f_1 - \gradt \bff_2  \mbox{ in } \O, \quad
u = 0 \mbox{ on } \p\O.
\eeq
We assume the following very mild conditions on the coefficients. 
The diffusion coefficient matrix $A \in L^{\infty}(\O)^{d\times d}$ is a given $d\times d$ tensor-valued function;  the matrix $A$ is uniformly symmetric positive definite: there exist positive constants $0 < \Lambda_0 \leq \Lambda_1$ such that
$
\Lambda_0 \by^T\by \leq \by^T A \by \leq \Lambda_1 \by^T\by
$
for all $\by\in \mathbb{R}^d$ and almost all $x\in \O$. 
The coefficients $\bb \in L^{\infty}(\O)^d$ and  $c\in L^{\infty}(\O)$ are given vector- and scalar-valued bounded functions, respectively.

We assume that the right-hand side is in $H^{-1}(\O):= (H_0^1(\O))'$. As discussed in \cite{Evans:10,Brezis:11}, any functional in $H^{-1}(\O)$ can be written as $f_1 - \gradt \bff_2$ for $f_1\in L^2(\O)$ and $\bff_2\in L^2(\O)^d$. 
Here  $f_1\in L^2(\O)$ and $\bff_2 \in L^2(\O)^d$ are given functions.  We have $f_1 -\gradt \bff_2 \in H^{-1}(\O)$. The divergence of $\bff_2$ should be understood in the distributional sense, i.e., for a $\bff_2\in L^2(\O)^d$, its divergence $\gradt \bff_2 \in H^{-1}(\O)$ is defined as follows:
$$
(\gradt \bff_2, v) := -(\bff_2, \nabla v)\quad \forall v\in H^1_0(\O).
$$
Define the bilinear form 
$$
    a(w,v) := (A\nabla w,\nabla v) + (\bb\cdot\nabla w+c w,v) \quad \forall w,v\in H^1_0(\Omega).
$$
The corresponding weak problem of \eqref{pde1_m1} is to find $u\in H^1_0(\Omega)$, such that 
\begin{equation}
    a(u,v) = (f_1,v)+(\bff_2,\nabla v)\quad \forall v\in H^1_0(\Omega).
    \label{pde1_m1_weak}
\end{equation}
It is easy to see that the bilinear form $a(\cdot,\cdot)$ is continuous with respect to the $H^1$-norm. For simplicity, we assume that the problem \eqref{pde1_m1} or its weak problem \eqref{pde1_m1_weak} has a unique solution. This assumption means that the bilinear form $a(\cdot,\cdot)$ satisfies an inf-sup condition. Thus, there exist two positive constants $\beta$ and $C_{con}$, such that
\beq
        \beta \leq \inf_{w\in H^1_0(\Omega)}\sup_{v\in H^1_0(\Omega)} \frac{a(w,v)}{ \| w\|_1 \|v\|_1} \quad\mbox{and}\quad
        a(w,v) \leq C_{con} \|w\|_1 \|v\|_1\quad \forall w,v\in H^1_0(\Omega).
    \label{infsup}
\eeq
\begin{rem} \label{rem_infsupa}
For the bilinear form $a$, there are some simple cases that it is coercive on $H^1_0(\O)$. For example, when $c-\dfrac{1}{2}\gradt \bb \geq 0$, then $(\bb\cdot\nabla v+c v,v)=((c-\dfrac{1}{2}\gradt \bb)v,v) \geq 0$, then $a(v,v) \geq C\|v\|_1^2$. But, there are other cases that the equation may be indefinite. For example, when $\bb=\bzero$ and $c= - \kappa^2$ for some $\kappa>0$, we get the  Helmholtz equation. The Helmholtz equation is clearly not coercive. But, as long as \eqref{pde1_m1_weak} has only $u=0$ as its solution when the righthand side of  \eqref{pde1_m1_weak} is zero, the weak problem \eqref{pde1_m1_weak} is still well-posed. The inf-sup constant $\beta$ may depend on the coefficients of the PDE \eqref{pde1_m1}.
\end{rem}
The conforming finite element approximation problem of \eqref{pde1_m1_weak} is: Find $u_\bullet^c \in \bbV_\blt$, such that,
\beq \label{discrete_2ndorder}
a(u_\bullet^c,v) = (f_1,v) + (\bff_2, \nabla v) \quad \forall v \in \bbV_\blt.
\eeq
Without assuming extra regularity, based on the duality argument, Schatz and Wang \cite{SW:96} proved that \eqref{discrete_2ndorder} has a unique solution provided that the mesh size $h_\bullet$ of $\cT_\bullet$ is smaller than a fixed mesh-size $h_{\mathtt{fix}}$. We assume that the mesh-size function $h_0$ of the initial mesh $\cT_0$ is smaller than this $h_{\mathtt{fix}}$. Due to the facts that all $\cT_\bullet\in\bbT$ are generated from $\cT_0$ and that the mesh size is monotonically decreasing by {\bf Assumption R1}, we have the following discrete inf-sup stability:
\beq \label{infsup_dis_2nd}
         \beta_{0} \leq \inf_{w_\bullet \in \bbV_\blt} \sup_{v_\bullet \in \bbV_\blt} \frac{a(w_\bullet,v_\bullet)}{\| w_\bullet\|_1\| v_\bullet\|_1} \quad \forall \cT_\bullet \in \bbT.
\eeq
\begin{rem} \label{remakinfsup}
As seen from Lemma 3 of  \cite{SW:96},  the stability constant $\beta_{0} $ is uniform with respect to the mesh-size but may depend on the fixed mesh size $h_{\mathtt{fix}}$.
\end{rem}
\begin{rem}
In Proposition 1 of \cite{BHP:17}, without assuming any regularity of the problem, the authors showed that there is a finite element space $\bbV_0$ that is rich enough (for example, its mesh-size smaller that a fixed $h_{\mathtt{fix}}$) such that the inf-sup condition \eqref{infsup_dis_2nd} holds, and for all discrete spaces $\bbV_\blt\supseteq \bbV_0$, the inf-sup condition \eqref{infsup_dis_2nd} holds with the same constant.  
\end{rem}

\begin{rem}
If we assume some regularity, for example, $u\in H^{1+s}(\O)$ for some $s>0$, then similar to the computations in Lemmas 35.14 and 35.16 of \cite{FE2}, the mesh-size $h_{\mathtt{fix}}$ can be explicitly computed.  
\end{rem}

\begin{rem}
As discussed in Remark \ref{rem_infsupa}, the bilinear form $a(\cdot,\cdot)$ is coercive for some simple cases, then there will be no  requirement on the mesh size. 
\end{rem}
We have the error equation:
$$
a(u-u_\bullet^c, v_\bullet) = 0 \quad \forall v_\bullet \in \bbV_\blt.
$$
The following best approximation holds (see \cite{XZ:03}),
\beq \label{apriori_2nd}
\|u-u_\bullet^c\|_1 \leq \dfrac{C_{con}} {\beta_{0} } \inf_{v\in \bbV_\bullet} \|u-v\|_1.
\eeq
Combining the approximation properties of \eqref{nodal_inter} and \eqref{sz}, we have an almost localized a priori error estimate with respect to local regularity.
\begin{thm}\label{thm_uapp}
Assume that $u\in H^{1}_0(\O)$,  $u|_T \in H^{1+s_T}(T)$ for $T\in \cT_{\blt,s}$,  and  $u|_{\Delta_T} \in H^{1+s_{\Delta_T}}({\Delta_T})$ for $T\in \cT_\blt\backslash\cT_{\blt,s}$, where $\max_{T\in\cT_{\blt,s}}\{s_T\}\leq 2$ and $\max_{T\in\cT_\blt \backslash\cT_{\blt,s}}\{s_{\Delta_T}\}\leq 2$,
\beq\label{uapp}
\|u-u_\bullet^c\|_1 \leq \dfrac{C_{con}} {\beta_{0} }
\inf_{v\in \bbV_\blt} \|u-v\|_1 \leq C (\sum_{T\in\cT_{\blt,s}}h_T^{s_T} |u|_{1+s_T,T}+\sum_{T\in\cT_\blt\backslash \cT_{\blt,s}}h_T^{s_{\Delta_T}} |u|_{1+s_{\Delta_T},\Delta_T}).
\eeq
\end{thm}

\begin{rem} We do not consider the special case that the coefficient $A$ is very small in this paper. 
 The standard conforming finite element method may not be the most suitable choice, and additional specialized constructions and analyses would be required. One may seek solve it by an SUPG method, and construct a specially designed least-squares functional error estimator.
\end{rem}

\section{Least-Squares Finite Element Method and its Built-in Least-Squares Functional Error Estimator}
\setcounter{equation}{0}
In this section, we give a brief introduction of the standard least-squares finite element method and its built-in least-squares functional error estimator.

For the second-order elliptic equation \eqref{pde1_m1}, let the flux $\bsigma = \bff_2 -A \nabla u$. We have the  first-order system:
\beq \label{fosys1}
\left\{
\begin{array}{rclll}
\bsigma + A \nabla u &=&\bff_2& \mbox{ in } \O,
\\[1mm]
\gradt\bsigma +\bb\cdot \nabla u+c u  &=&f_1&  \mbox{ in } \O,\\
u &=& 0 &\mbox{ on } \p\O.
\end{array}
\right.
\eeq
For $u\in H^1_0(\O)$, we have  $\bsigma = -A \nabla u +\bff_2 \in L^2(\O)^d$ and $\gradt\bsigma  = f_1 - \bb\cdot \nabla u - c u\in L^2(\O)$, so $(\bsigma,u)\in  \bbX = H(\divvr;\O)\times H_0^1(\O)$. Note that we also have $\nabla u = A^{-1}(\bff_2-\bsigma)$, then $\bb\cdot\nabla u$ can be written as a linear combination of $\bb\cdot\nabla u$ and $\bb\cdot A^{-1}(\bff_2-\bsigma)$. Thus, the second equation in \eqref{fosys1} can also be written as
$$
\gradt\bsigma +B(\bsigma,u;\gamma) =g(\gamma)
$$
where
$
B(\btau,v;\gamma) = \gamma \bb\cdot\nabla v-(1-\gamma)\bb\cdot A^{-1}\btau+cv
\quad\mbox{and}\quad
g(\gamma) = f_1 -(1-\gamma)\bb\cdot A^{-1}\bff_2.
$
Here $\gamma \in [0,1]$ is a fixed number.
Thus, we have an equivalent and more general first-order system,
\beq \label{fosys2}
\left\{
\begin{array}{rclll}
\bsigma + A \nabla u &=&\bff_2& \mbox{ in } \O,
\\[1mm]
\gradt\bsigma +B(\bsigma,u;\gamma) &=& g(\gamma)&\mbox{ in } \O,\\
u &=& 0 &\mbox{ on } \p\O.
\end{array}
\right.
\eeq
For $(\btau,v)\in \bbX$, define least-squares functional for the system \eqref{fosys2},
\begin{eqnarray} \label{LS_functional}
\LS(\btau,v;f_1,\bff_2,\gamma) &:=& \|A^{-1/2}\btau + A^{1/2}\nabla v-A^{-1/2}\bff_2\|_0^2 +\|\gradt\btau +B(\btau,v;\gamma)-g(\gamma) \|_0^2.
\end{eqnarray}
The corresponding least-squares minimization problem is:
\begin{eqnarray}
\mbox{Find } (\bsigma,u)\in \bbX  &\mbox{ such that }&
\LS(\bsigma,u;f_1,\bff_2,\gamma) = \inf_{(\btau,v)\in \bbX}\LS(\btau,v;f_1,\bff_2,\gamma) . 
\end{eqnarray}
The Euler-Lagrange weak problem is: Find $(\bsigma,u)\in  \bbX$, such that
\beq\label{lsvp}
b_{ls}((\bsigma,u), (\btau,v);\gamma) = F_{ls}(\btau,v;\gamma) \quad \forall(\btau,v)\in  \bbX,
\eeq
where the bilinear form $b_{ls}$ and the linear form $F_{ls}$ are defined for all  $(\brho,w)$ and $(\btau,v)\in \bbX$ as:
\begin{eqnarray*}
b_{ls}((\brho,w), (\btau,v);\gamma) &=& (\brho+A\nabla w, A^{-1}\btau+\nabla v)+ (\gradt\brho +B(\brho,w;\gamma),\gradt\btau +B(\btau,v;\gamma)) ,\\[1mm]
F_{ls}(\btau,v;\gamma) &=& (\bff_2,  A^{-1}\btau+\nabla v)+ (g(\gamma) ,\gradt\btau  +B(\btau,v;\gamma)).
\end{eqnarray*}
For the case $\gamma=1$, the following norm equivalence is well-known,
\beq
\label{ls_equvalience}
C_1\tri (\btau,v )\tri^2 \leq \LS(\btau,v; 0,\bzero;1) \leq C_2\tri (\btau,v )\tri^2 \quad
 \forall(\btau,v)\in  \bbX,
\eeq
for some positive constants $C_1$ and $C_2$. Various proofs of the norm-equivalence \eqref{ls_equvalience} (mainly the coercivity of $b_{ls}$ which is equivalent to the first inequality of \eqref{ls_equvalience}) of the general second-order elliptic equation with the simple assumption that the original weak problem has a unique solution can be found in \cite{CLMM:94,BLP:97,Cai:04,Ku:07,CZ:10b,CFZ:15}. A recent simple proof can be found in \cite{Zhang:23}. 
Here, we present a proof for the general case $\gamma\in [0,1]$ in the sprit of Theorem 3.1 of \cite{Zhang:23}.
\begin{thm}
The following norm equivalence is true for  $\gamma\in [0,1]$:
\beq
\label{ls_equ_gamma}
C_{ls,coe}\tri (\btau,v )\tri^2 \leq \LS(\btau,v; 0,\bzero;\gamma) \leq C_{ls,con}\tri (\btau,v )\tri^2 \quad  \forall(\btau,v)\in  \bbX,
\eeq
for some positive constants $C_{ls,coe}$ and $C_{ls,con}$. 
\end{thm}
\begin{proof}
By the integration by parts, we have $(\btau, \nabla w)+(\gradt\btau,w) =0$ for all $(\btau,w) \in \bbX$. 
Then for $\btau \in H(\divvr;\O)$ and $v$ and $w$ in $H_0^1(\O)$, we have
\begin{eqnarray} \nonumber
a(v,w) &=& (A\nabla v, \nabla w)+ (\bb\cdot\nabla v+c v,w)  =  (A\nabla v+\btau, \nabla w)+(\gradt\btau+ \bb\cdot\nabla v+c v,w) \\ \nonumber 
&=&  (A\nabla v+\btau, \nabla w)+(\gradt\btau+ B(\btau,v;\gamma),w) + (1-\gamma)(\nabla v+A^{-1}\btau, \bb w) \\ \label{vww}
&=&  (\nabla v+A^{-1}\btau, A\nabla w+(1-\gamma)\bb w)+(\gradt\btau+ B(\btau,v;\gamma),w).
\end{eqnarray}
It follows from \eqref{infsup}, \eqref{vww}, the Cauchy-Schwarz and Poincar\'e inequalities, the assumption on $A$ and $\bb$, for any $(\btau,v) \in \bbX$, 
\begin{eqnarray*}
\beta \|v\|_1 &\leq& \sup_{w\in H_0^1(\O)} \dfrac{a(v,w)}{\|w\|_1} 
=\sup_{w\in H_0^1(\O)}\dfrac{ (\nabla v+A^{-1}\btau, A\nabla w+(1-\gamma)\bb w)+(\gradt\btau+ B(\btau,v;\gamma),w) }{\|w\|_1} \\
&\leq & C (\|A^{1/2}\nabla v+A^{-1/2}\btau\|_0 + \|\gradt\btau +  B(\btau,v;\gamma)\|_0) \leq C\, \LS(\btau,v; 0,\bzero;\gamma)^{1/2}.
\end{eqnarray*}
By the triangle inequality, we have
$$
\|\btau\|_0 \leq \|\btau+A\nabla v\|_0+\|A\nabla v\|_0 \leq C(\|A^{1/2}\nabla v+A^{-1/2}\btau\|_0 + \|v\|_1)\leq C \LS(\btau,v; 0,\bzero;\gamma)^{1/2}.
$$
Using the triangle inequality and the fact that $\|B(\btau,v;\gamma)\|_0  \leq C (\|\btau\|_0+\|v\|_1)$, we have
$$
\|\gradt\btau\|_0 \leq \|\gradt\btau +  B(\btau,v;\gamma)\|_0+\|B(\btau,v;\gamma)\|_0 \leq C(\|\gradt\btau +  B(\btau,v;\gamma)\|_0 + \|\btau\|_0+\|v\|_1)\leq C \LS(\btau,v; 0,\bzero;\gamma)^{1/2}.
$$
Thus, $C_{ls,coe}\tri (\btau,v )\tri^2 \leq \LS(\btau,v; 0,\bzero;\gamma)$ is true for $(\btau,v) \in \bbX$.

The other bound is relatively easy. By the definition of the $\LS$ functional, the triangle inequality, and the fact that $\|B(\btau,v;\gamma)\|_0  \leq C (\|\btau\|_0+\|v\|_1)$ for $(\btau,v) \in \bbX$,  we immediately have
\begin{eqnarray*}
\LS(\btau,v; 0,\bzero;\gamma)^{1/2} &\leq& C (\|A^{1/2}\nabla v+A^{-1/2}\btau\|_0 + \|\gradt\btau +  B(\btau,v;\gamma)\|_0) \leq C\tri (\btau,v )\tri \quad \forall(\btau,v) \in \bbX.
\end{eqnarray*}
The theorem is proved.
\end{proof}
\begin{rem}
The first inequality in \eqref{ls_equ_gamma} is equivalent to the coercivity of the bilinear form $b_{ls}$ in $\bbX$, and the second inequality in \eqref{ls_equ_gamma} is equivalent to the continuity of the bilinear form $b_{ls}$ in $\bbX$.
\end{rem}
\begin{rem}
We can choose different weights for the two terms in the least-squares functional \eqref{LS_functional}. The second term can be $\|\kappa(\gradt\btau +B(\btau,v;\gamma)-g(\gamma))\|_0^2$, with a positive function $\kappa$. We use one here for simplicity.
\end{rem}
The least-squares finite element problem in $\bbX_\blt$ is:
\begin{eqnarray*}
\mbox{Find } (\bsigma_\blt^{ls},u_\blt^{ls})\in \bbX_\blt  &\mbox{ such that }&
\LS(\bsigma_\blt^{ls},u_\blt^{ls};f_1,\bff_2,\gamma) = \inf_{(\btau,v)\in \bbX_\blt}\LS(\btau,v;f_1,\bff_2,\gamma). 
\end{eqnarray*}
Or, equivalently: Find $(\bsigma_\blt^{ls},u_\blt^{ls})\in \bbX_\blt$, such that
\beq\label{lsfem1h}
b_{ls}((\bsigma_\blt^{ls},u_\blt^{ls}), (\btau,v);\gamma) = F_{ls}(\btau,v;\gamma) \quad \forall(\btau,v)\in  \bbX_\blt.
\eeq
The a priori error estimate for this problem with an $H^{-1}$-righthand side of $\gamma = 1$ can be found in Theorem 4.1 of \cite{Zhang:23}. For other $\gamma\in [0,1]$, the same a priori error estimate holds due to the continuity and coercivity of the bilinear form $b_{ls}$.

We introduce some notations of the built-in least-squares functional error estimator associated a mesh $\cT_\bullet \in \bbT$. Let  $v_\blt \in \bbV_\blt$  and $\btau_\blt \in \bbW_\blt$ be two arbitrary finite element functions in their spaces associated with the mesh  $\cT_\bullet$, respectively.  For an element $T\in \cT_\blt$, we define the element-wise least-squares functional error indicator as
\begin{eqnarray}\label{ap_indicator}
\eta^{ls}_\bullet(T;\btau_\bullet,v_\bullet) &:=& \left(\|A^{-1/2}\btau_\bullet + A^{1/2}\nabla v_\bullet-A^{-1/2}\bff_2\|_{0,T}^2 +\|\gradt\btau_\bullet +B(\btau_\blt,v_\bullet;\gamma)-g(\gamma) \|_{0,T}^2\right)^{1/2}.
\end{eqnarray}
Here, for simplicity, we omit the parameter $\gamma$ in the notation of $\eta^{ls}_\bullet$.
For a collection of elements $\cU_\bullet\subset \cT_\bullet$, the least-squares functional a posteriori error estimator defined on $\cU_\bullet$ is
\beq \label{ap_indicator_cU}
\eta^{ls}_\bullet(\cU_\bullet; \btau_\bullet,v_\bullet):= \left( \sum_{T\in\cU_\blt}
\eta^{ls}_\bullet(T;\btau_\bullet,v_\bullet)^2\right)^{1/2}.
\eeq
For the case $\cU_\blt = \cT_\blt$, we use a simpler notation,
\beq \label{ap_estimator}
\eta^{ls}_\bullet(\btau_\bullet,v_\bullet) := \eta^{ls}_\bullet(\cT_\bullet; \btau_\bullet,v_\bullet).
\eeq
It is easy to see that 
\beq \label{lsap}
\eta^{ls}_\bullet(\btau_\bullet,v_\bullet) = \LS(\btau_\bullet,v_\bullet;f_1,\bff_2,\gamma)^{1/2} 
\eeq
Using the facts that $\bff_2=A \nabla u +\bsigma$ and $g(\gamma)=\gradt\bsigma+B(\bsigma,u;\gamma)$ from \eqref{fosys2}, we have the following identity:
\begin{eqnarray*}
\LS(\btau_\bullet,v_\bullet;f_1,\bff_2,\gamma) &=&\|A^{-1/2}\btau_\bullet + A^{1/2}\nabla v_\bullet -A^{-1/2}\bff_2\|_0^2 +\|\gradt\btau_\bullet +B(\btau_\blt,v_\bullet;\gamma)-g(\gamma) \|_0^2\\  
&=& \|A^{-1/2}(\bsigma- \btau_\bullet) + A^{1/2}\nabla (u-v_\bullet)\|_0^2 +\|\gradt(\bsigma- \btau_\bullet)  +B(\bsigma-\btau_\blt,u-v_\bullet;\gamma)\|_0^2  \\
&=& \LS(\bsigma- \btau_\bullet, u-v_\bullet;0,\bzero,\gamma). 
\end{eqnarray*}
By  \eqref{ls_equ_gamma},  the following reliability and efficiency bounds are true,
\begin{eqnarray} \label{ls_rel_eff1}
C_{ls,coe} \tri(\bsigma- \btau_\bullet,u-v_\bullet)\tri ^2 &\leq & \LS(\bsigma- \btau_\bullet,u-v_\bullet;0,\bzero,\gamma) = \LS(\btau_\bullet,v_\bullet;f_1,\bff_2,\gamma)  \\ \label{ls_rel_eff2}
&= &\eta_\bullet^{ls}(\btau_\bullet,v_\bullet)^2\leq C_{ls,con} \tri(\bsigma- \btau_\bullet,u-v_\bullet)\tri ^2.
\end{eqnarray}
An important fact of \eqref{ls_rel_eff1} and  \eqref{ls_rel_eff2} is that $(\btau_\bullet,v_\bullet) \in \bbX$ does not need to be the numerical solution of a specific problem, say the LSFEM problem \eqref{lsfem1h}. In fact, we can even relax the condition to let the pair to be any functions in $\bbX$. We restrict them in their finite element spaces $\bbX_\blt$ associated with $\cT_\bullet$ for simplicity of the presentation only. As we can see from the plain convergence proof will be presented later (or \cite{FP:20} for the LSFEM case), if the pair $(\btau_\bullet,v_\bullet)$ is not a good approximation of the true solution, then the adaptive algorithm is not going to converge. In other words, to have a good adaptive numerical method, we need both a priori and a posteriori error analysis to be valid.

Let  $(\bsigma_\blt^{ls},u_\blt^{ls})\in \bbX_\blt $ be the numerical solution of the LSFEM problem \eqref{lsfem1h}, we immediately have the reliability and efficiency of the least-squares functional error estimator for the LSFEM approximation \eqref{lsfem1h}, 
\beq \label{lsfem_rel_eff}
\sqrt{C_{ls,coe}} \tri(\bsigma - \bsigma_\blt^{ls},u-u_\blt^{ls})\tri \leq  \eta^{ls}_\blt(\bsigma_\blt^{ls},u_\blt^{ls}) \leq \sqrt{C_{ls,con}} \tri(\bsigma - \bsigma_\blt^{ls},u-u_\blt^{ls})\tri.
\eeq

\section{Non-Intrusive Least-Squares Functional Error Estimator for Conforming FEM of Elliptic Equation}
\setcounter{equation}{0}
In this section. we develop a posteriori error estimator for the discrete problem \eqref{discrete_2ndorder} using the least-squares functional estimator. With  the solution of \eqref{discrete_2ndorder} $u_\bullet^c\in\bbV_\blt$ available, in order to use the least-squares functional error estimator, we need to construct a $\bsigma_\bullet\in \bbW_\blt$. The simplest way to get a good $\bsigma_\bullet$ is to replace the function $v$ in the least-squares functional \eqref{LS_functional} by $u$'s approximation $u_\bullet^c\in\bbV_\blt\subset H^1_0(\O)$. We get a new functional,
$$
\JJ(\btau; u_\bullet^c, f_1,\bff_2,\gamma) := \LS(\btau,u_\bullet^c;f_1,\bff_2,\gamma)= \|A^{-1/2}\btau + A^{1/2}\nabla u_\bullet^c-A^{-1/2}\bff_2\|_0^2 +\|\gradt\btau +B(\btau,u_\bullet;\gamma)-g(\gamma) \|_0^2.
$$
Then the corresponding minimization problem to find $\bsigma_\bullet^{r}\in \bbW_\blt$ is:
$$
\mbox{Find } \bsigma_\bullet^{r} \in \bbW_\blt  \mbox{ such that }
\JJ(\bsigma_\bullet^{r};u_\bullet^c, f_1, \bff_2,\gamma) = \inf_{\btau\in \bbW_\blt} \JJ(\btau;u_\bullet^c, f_1, \bff_2,\gamma). 
$$
Or, equivalently: find $\bsigma_\bullet^{r} \in \bbW_\blt$, such that,
\begin{eqnarray} \label{flux_recovery_discrete}
b_{r}(\bsigma_\bullet^{r},\btau;\gamma) = F_r(\btau;u_\bullet^c,\gamma)\quad\forall \btau\in \bbW_\blt,
\end{eqnarray}
where the bilinear and linear forms are defined for every $\brho\in H(\divvr;\O)$, $\btau\in H(\divvr;\O)$, and $v\in H^1_0(\O)$ as follows,
\begin{eqnarray*} 
b_{r}(\brho,\btau;\gamma)&:=& (A^{-1}\brho,\btau) + (\gradt\brho-(1-\gamma)\bb\cdot A^{-1}\brho,\gradt\btau-(1-\gamma)\bb\cdot A^{-1}\btau),
\\ 
F_r(\btau;v,\gamma)&:=& (A^{-1}\bff_2-\nabla v,\btau) + (f_1 -(1-\gamma)\bb\cdot A^{-1}\bff_2-\gamma\bb\cdot\nabla v-cv,\gradt\btau-(1-\gamma)\bb\cdot A^{-1}\btau).
\end{eqnarray*}
We define the following continuous problem: Find $\bsigma\in H(\divvr;\O)$, such that,  
\beq \label{flux_recovery}
b_{r}(\bsigma,\btau;\gamma) = F_r(\btau;u,\gamma)\quad \forall \btau\in H(\divvr;\O).
\eeq
Note that the exact solution $(\bsigma, u)$ of the original first-order system \eqref{fosys1} is the solution of \eqref{flux_recovery}.
The problem \eqref{flux_recovery_discrete} can be viewed as a finite element approximation of the continuous problem \eqref{flux_recovery} with the exact solution $u$ replaced by its approximation $u_\bullet^c$.

Both the problems \eqref{flux_recovery} and \eqref{flux_recovery_discrete} are uniquely solvable. One way to check the coercivity of \eqref{flux_recovery} and \eqref{flux_recovery_discrete} is from the norm-equivalence \eqref{ls_equ_gamma}. Let $v=0$ in \eqref{ls_equ_gamma}, then 
$$
C_{ls,coe}\|\btau\|_{H(\divvr;\O)}^2 =C_{ls,coe}\tri (\btau,0 )\tri^2 \leq \LS(\btau,0; 0,\bzero,\gamma) \quad
 \forall \btau \in  H(\divvr;\O).
$$
 
\begin{thm}
The following a priori estimate for $\bsigma_\bullet^{r}\in \bbW_\blt$ as an approximation of $\bsigma$ of \eqref{flux_recovery} or \eqref{fosys1} is true:
\begin{eqnarray} \label{apriori_sigma}
\|\bsigma-\bsigma_\bullet^{r}\|_{H(\divvr)} \leq  C \left( \inf_{\btau \in \bbW_\bullet} \|\bsigma-\btau\|_{H(\divvr)} +\|u-u_\blt^c\|_1 \right)\leq  C\inf_{(\btau,v)\in \bbX_\blt}\tri(\bsigma-\btau, u-v)\tri.
\end{eqnarray}
Assume that $u$ has the regularity assumptions in Theorem \ref{thm_uapp} and further assume that $\bsigma|_T \in H^{\ell_T}(T)$ and $\gradt\bsigma|_T \in H^{t_T}(T)$,  for $1\leq \ell_T \leq 2$ and $1\leq t_T \leq 2$. Then the following a priori error estimate holds:
\begin{eqnarray} \label{apriori_rsigma}
\|\bsigma-\bsigma_\bullet^{r}\|_{H(\divvr)} &\leq &C (\sum_{T\in\cT_{\blt,s}}h_T^{s_T} |u|_{1+s_T,T}+\sum_{T\in\cT_\blt\backslash \cT_{\blt,s}}h_T^{s_{\Delta_T}} |u|_{1+s_{\Delta_T},\Delta_T} \\ \nonumber
&&+\sum_{T\in\cT_\blt}(h_T^{\ell_T} |\bsigma|_{\ell_T,T} + h_T^{t_T} |\gradt\bsigma|_{t_T,T})).
\end{eqnarray}
\end{thm}
\begin{proof}
The proof of \eqref{apriori_sigma} with $\gamma=1$ can be found in Theorem 4.1 of \cite{CZ:10b}. The general case with $\gamma \in [0,1]$ can be proved similarly. The a priori estimate \eqref{apriori_rsigma} is a simple result of the best approximation result \eqref{apriori_sigma}, the a priori estimate \eqref{uapp}, and the approximation results \eqref{RT_inter1} and \eqref{RT_inter2}. 
\end{proof}
From \eqref{apriori_sigma}, we should choose the approximation space $\bbW_\blt$ for $\bsigma$ in compatible with $\bbV_\blt$ of $u$ to keep the approximation order optimal.

Let $\cT_\bullet \in \bbT$. We have computed $u_\bullet^c \in \bbV_\blt$ from the conforming finite element discrete problem \eqref{discrete_2ndorder} and recovered a numerical flux $\bsigma_\bullet^{r} \in \bbW_\blt$ from \eqref{flux_recovery_discrete}. In view of the least-squares functional error estimator, we have all the ingredients. Thus, we can use the least-squares functional error estimator $\eta^{ls}_\blt(\bsigma_\bullet^{r},u_\bullet^c)$ defined in \eqref{ap_estimator} (and its local contributions defined in \eqref{ap_indicator}, \eqref{ap_indicator_cU}). 
%
If we view $\eta^{ls}_\blt(\bsigma_\blt^{r},u_\blt^{c})$ as the estimator for the $u$-problem \eqref{discrete_2ndorder} only, the following reliability bound is a simple consequence of \eqref{ls_rel_eff1}, 
\beq
\|u-u_\blt^c\|_1 \leq \tri(\bsigma - \bsigma_\blt^{r},u-u_\blt^{c})\tri \leq  \dfrac{1}{\sqrt{C_{ls,coe}}}\eta^{ls}_\blt(\bsigma_\blt^{r},u_\blt^{c}).
\eeq
On the other hand, the efficiency bound $\eta^{ls}_\blt(\bsigma_\blt^{r},u_\blt^{c}) \leq C \|u-u_\blt^c\|_1$ is not easy to prove, see the discussion in Remark \ref{remeff}. Thus, we seek an alternative view to see the error estimator $\eta^{ls}_\blt(\bsigma_\blt^{r},u_\blt^{c})$ not as an error estimator for the original problem but as the estimator for a combined two-step problem.

\begin{rem}
The extra computational cost of constructing the a posteriori error estimator is solving the global problem \eqref{flux_recovery_discrete}. Compared to the standard LSFEM, where both $u$ and $\bsigma$ is solved, the approach here has a similar computational cost. As discussed and implemented in \cite{CZ:10b}, if we seek optimal computational cost, we can use well-studied multigrid method for $H(\divvr)$ problem to solve \eqref{flux_recovery_discrete}.
\end{rem}

\section{Alternative view on Non-Intrusive Least-Squares Functional Error Estimator}
\setcounter{equation}{0}

In this section, we present an alternative view on the two-step (solve-recover) \eqref{discrete_2ndorder}-\eqref{flux_recovery_discrete} procedure of the least-squares functional error estimator. We rewrite it as an equivalent combined problem solving both the flux $\bsigma$ and solution $u$.

\noindent{\bf Combined two-step problem.}
Consider the following problem: find $(\bsigma,u) \in \bbX$ such that
\begin{equation} \label{combined_weak}
    \cA_{2s}((\bsigma,u), (\btau,v);\gamma) = G_{2s}(\btau,v;\gamma) \quad \forall (\btau,v) \in \bbX,
\end{equation}
where for all $(\brho,w),(\btau,v)\in \bbX$, the combined bilinear form $\cA_{2s}$ and linear form $G_{2s}$ are defined by 
\begin{eqnarray}
    \cA_{2s}((\brho,w),(\btau,v);\gamma)&:=& a(w,v) + (A^{-1}\brho+\nabla w,\btau) + (\gradt\brho+B(\brho,w;\gamma),\gradt\btau- (1-\gamma)\bb\cdot A^{-1}\btau),\\ 
    G_{2s}(\btau,v;\gamma)&:=& 
    (f_1,v)+ (\bff_2,\nabla v+A^{-1}\btau)
    +(f_1 - (1-\gamma)\bb\cdot A^{-1}\bff_2, \gradt\btau - (1-\gamma)\bb\cdot A^{-1}\btau). 
\end{eqnarray}
Taking $\btau=0$ in \eqref{combined_weak}, we get the original weak problem \eqref{pde1_m1_weak}. Thus $u$ can be obtained from \eqref{combined_weak} without invoking $\bsigma$ or $\btau$. After obtaining $u$, let $v=0$  in \eqref{combined_weak}, we get the least-squares flux recovery problem \eqref{flux_recovery}. Then, the combined two-step problem \eqref{combined_weak} is equivalent to the two-step process, \eqref{pde1_m1_weak} and \eqref{flux_recovery}: their solutions are identical. 

Now consider the finite element approximation of  \eqref{combined_weak} in $\bbX_\blt$, we have the following discrete combined two-step problem.

\noindent{\bf Finite element approximation of the combined two-step problem.}
Find $(\bsigma_\bullet^{r},u_\bullet^c) \in \bbX_\blt$ such that
\begin{equation} \label{combined_discrete}
    \cA_{2s}((\bsigma_\bullet^{r},u_\bullet^c), (\btau,v);\gamma) = G_{2s}(\btau,v;\gamma)\quad \forall (\btau,v) \in \bbX_\blt.
\end{equation}
Similarly, taking $\btau=0$ in \eqref{combined_discrete}, we get the original discrete problem \eqref{discrete_2ndorder}. After obtaining $u_\bullet^c$, let $v=0$ in  \eqref{combined_discrete}, we get the discrete least-squares flux recovery problem \eqref{flux_recovery_discrete}. Thus, the combined discrete problem \eqref{combined_discrete} is equivalent to the two-step solve-recovery process, \eqref{discrete_2ndorder} and \eqref{flux_recovery_discrete}: both get the same discrete solutions $\bsigma_\bullet^{r}$ and $u_\bullet^c$. Thus, in stead of view the least-squares functional estimator as an a posteriori error estimator for the original conforming finite element approximation \eqref{discrete_2ndorder}, we  view it as a posteriori error estimator for the combined problem \eqref{combined_weak} and \eqref{combined_discrete}. 

As a standard step of numerical analysis, we discuss a priori and a posteriori error estimates for the combined problem.
The unique solvability of the combined problem is in fact the result of the well-posedness of two sub-problems. To fit our problem into the setting of the framework of plain convergence in \cite{Sie:11}, we prove the following theorem on the inf-sup stability for the bilinear form $\cA_{2s}$.
\begin{thm} \label{infsup-cm}
Assume that the mesh-size function $h_0$ of the initial mesh $\cT_0$ is smaller than the fixed mesh size $h_{\mathtt{fix}}$.
There exists a constant $\beta_{cb}>0$, 
which  is uniform with respect to the mesh-size but may depend on the fixed mesh size $h_{\mathtt{fix}}$ and the coefficients $A$, $\bb$, $c$ of the problem \eqref{pde1_m1}, such that 
\beq \label{infsup_cb}
\beta_{cb} \leq \inf_{(\brho_\bullet,w_\bullet)\in \bbX_\bullet }\sup_{(\btau_\bullet,v_\bullet)\in \bbX_\bullet} \dfrac{ \cA_{2s}((\brho_\bullet,w_\bullet),(\btau_\bullet,v_\bullet);\gamma)}{\tri(\brho_\bullet,w_\bullet)\tri \tri(\btau_\bullet,v_\bullet)\tri}
\quad \forall \cT_\bullet \in \bbT.
\eeq
\end{thm}
\begin{proof}
From the definition of $\cA_{2s}$, it is easy to see that 
\beq \label{eqeq}
 \cA_{2s}((\brho,w),(\bzero,v);\gamma) = a(w,v) \quad \forall (\brho,w) \in \bbX, \,\, v\in H^1_0(\O).
\eeq
From \eqref{infsup_dis_2nd} and \eqref{eqeq}, we have 
\beq
\beta_{0} \|w_\bullet\|_1 \leq \sup_{v_\bullet\in\bbV_\blt} \dfrac{a(w_\bullet,v_\bullet)}{\|v_\bullet\|_1} = 
\sup_{(\bzero,v_\bullet)\in\bbX_\blt} \dfrac{\cA_{2s}((\brho_\bullet,w_\bullet),(\bzero,v_\bullet);\gamma)}{\tri(\bzero,v_\bullet)\tri} 
\leq 
\sup_{(\btau_\bullet,v_\bullet)\in\bbX_\blt} \dfrac{\cA_{2s}((\brho_\bullet,w_\bullet),(\btau_\bullet,v_\bullet);\gamma)}{\tri(\btau_\bullet,v_\bullet)\tri}. 
\eeq
By the definition of $H(\divvr)$-norm, definition of $\cA_{2s}$, the fact that 
$(\nabla w,\btau) + (\gamma\bb\cdot\nabla w+c w,\gradt \btau- (1-\gamma)\bb\cdot A^{-1}\btau) + a(w,v) \leq C_1 \|w\|_1 \tri (\btau,v)\tri$ for $w\in H^1_0(\O)$  and $(\btau,v)\in \bbX$, we have
\begin{eqnarray*}
&&C \|\brho_\bullet\|_{H(\divvr)} \leq \sup_{\btau_\bullet \in \bbW_\bullet} 
\dfrac{(A^{-1}\brho_\bullet,\btau_\bullet) + (\gradt\brho_\bullet- (1-\gamma)\bb\cdot A^{-1}\brho_\bullet,\gradt\btau_\bullet- (1-\gamma)\bb\cdot A^{-1}\btau_\bullet )}{\|\btau_\bullet\|_{H(\divvr)}} \\
&&=   \sup_{(\btau_\bullet, v_\bullet)\in \bbX_\bullet} \dfrac{(A^{-1}\brho_\bullet,\btau_\bullet) + (\gradt\brho_\bullet- (1-\gamma)\bb\cdot A^{-1}\brho_\bullet,\gradt\btau_\bullet- (1-\gamma)\bb\cdot A^{-1}\btau_\bullet )}{\tri(\btau_\bullet,v_\bullet)\tri} \\ 
&&=   \sup_{(\btau_\bullet,v_\bullet)\in \bbX_\bullet} \dfrac{ \cA_{2s}((\brho_\bullet,w_\bullet),(\btau_\bullet,v_\bullet);\gamma)-  (\nabla w_\bullet,\btau_\bullet) - (\gamma\bb\cdot\nabla w_\bullet+c w_\bullet,\gradt \btau_\bullet- (1-\gamma)\bb\cdot A^{-1}\btau_\bullet) - a(w_\bullet,v_\bullet)} {\tri(\btau_\bullet,v_\bullet)\tri} \\
&&\leq   \sup_{(\btau_\bullet,v_\bullet)\in \bbX_\bullet} \dfrac{ \cA_{2s}((\brho_\bullet,w_\bullet),(\btau_\bullet,v_\bullet);\gamma)}{\tri(\btau_\bullet,v_\bullet)\tri} + C_1\|w_\bullet\|_1.
\end{eqnarray*}

Combined the above two results, we have the theorem.
\end{proof}
With Theorem \ref{infsup-cm}, we immediately have the following a priori error estimate for the combined two-step problem \eqref{combined_discrete} under the assumption that the mesh-size function $h_0$ of the initial mesh $\cT_0$ is smaller than the fixed mesh size $h_{\mathtt{fix}}$:
\begin{eqnarray} \label{ap_combined}
\tri(\bsigma-\bsigma_\bullet^{r}, u-u_\bullet^c)\tri &\leq &  C\inf_{(\btau,v)\in \bbX_\blt}\tri(\bsigma-\btau, u-v)\tri.
\end{eqnarray}
From the reliability and efficiency of the general least-squares functional error estimator \eqref{ls_rel_eff1}-\eqref{ls_rel_eff2}, we have following bounds for the a posteriori error estimator,
\beq \label{ls_rel_eff_combined}
\sqrt{C_{ls,coe}} \tri(\bsigma-\bsigma_\bullet^{r},u-u_\bullet^c)\tri  \leq \eta^{ls}_\blt(\bsigma_\bullet^{r},u_\bullet^c) \leq \sqrt{C_{ls,con} } \tri(\bsigma-\bsigma_\bullet^{r},u-u_\bullet^c)\tri.
\eeq
Since the general  least-squares functional error estimator \eqref{ls_rel_eff1}-\eqref{ls_rel_eff2} actually does not require any approximation properties of the discrete solutions, it is important to realize that we need both the a priori and a posteriori error controls to have a good adaptive numerical approximation.
\begin{rem}
The a priori result \eqref{ap_combined} is weaker than the individual results \eqref{apriori_2nd} and \eqref{apriori_sigma} since the a priori estimate of $u_\blt^c$ is actually an independent result.
\end{rem}

\begin{rem} \label{remeff}
Proving the standard efficiency result $\eta^{ls}_\blt(\bsigma_\bullet^{r},u_\bullet^c) \leq \|u-u_\bullet^c\|_1$ is indeed possible, but it may require additional work. In previous works such as \cite{CZ:09, CZ:10a, CZ:12}, to show the efficiency for two other related flux-recovery a posteriori error estimators $\eta_{m}$ based on a minimization problem, an explicit estimator $\eta_e$ is constructed, which is bigger than the minimized version $\eta_{m}\leq \eta_e$. Then for the explicit version, we show its equivalence to the standard residual-type error estimator, which is known to be efficient. 

In principle, a similar procedure could be applied to the least-squares functional error estimator. However, this approach is more challenging and goes against the initial idea of applying the estimator for not-well-studied problems. Instead, the idea of regarding the least-squares functional error estimator as an estimator for the combined two-step problems is simpler and more straightforward. By treating the solving and recovery process as one problem, the reliability and efficiency of the least-squares functional error estimator can be proved without the need for the complicated explicit equivalence construction used in other cases. This approach aligns with the initial motivation of applying the estimator to less-explored problems.
\end{rem}

\section{Plain Convergence of Adaptive Algorithm for Linear Problem}

\setcounter{equation}{0}

In this section, we will prove the plain convergence of the adaptive methods with the non-intrusive least-squares functional a posteriori error estimator in the sprit of \cite{Sie:11}.
\subsection{Marking Strategy}
Here, we use the same assumption on the marking strategy used in Section 2.2.4 of \cite{Sie:11} and Section 2.6 of \cite{FP:20}. The solution $(\bsigma_\bullet^{r},u_\bullet^c)\in \bbX_\blt$ is the discrete solution of the two-step combined problem \eqref{combined_discrete}.

\noindent
{\bf Assumption (M)} There exists a fixed function $g: [0,\infty) \rightarrow [0,\infty)$ being continuous at $0$, such that the set of marked elements $\cM_\blt \subset \cT_\blt$ (corresponding to the current approximation $(\bsigma_\bullet^{r},u_\bullet^c)$) satisfies that 
$$
\max_{T\in\cT_\blt\backslash\cM_\blt} \eta_\blt^{ls}(T;\bsigma_\bullet^{r},u_\bullet^c) \leq  g(\max_{T\in\cM_\blt} \eta_\blt^{ls}(T;\bsigma_\bullet^{r},u_\bullet^c)).
$$
In Section 2.2.4 of \cite{Sie:11} and Section 2.6 of \cite{FP:20}, it was discussed that for commonly used marking strategies such as the maximum strategy, the equilibration strategy, and the D\"orfler marking strategy, the marking assumption (M) is always satisfied with $g(s)=s$ and $\cM_\blt$ contains at least one element with a maximal error indicator.


\subsection{Adaptive Algorithm with Non-intrusive Least-Squares Functional Estimator}
We use the following adaptive algorithm.

\begin{algorithm}[H]
 \KwIn{Initial triangulation $\cT_0$.}

 \For{$\ell = 0,1,2,\cdots$ }{
  (i) {\bf Solve}. Compute the discrete solution $u_{\ell}^c\in \bbV_\ell$ by solving \eqref{discrete_2ndorder}.
  
  (ii){\bf Least-Squares Recovery}. Recover $\bsigma_{\ell}^{r}\in \bbW_\ell$ by solving the least-squares recovery problem \eqref{flux_recovery_discrete}. 
    
  (iii){\bf Estimate.} Compute $\eta_{\ell}^{ls}(T;\bsigma^r_\ell,u^c_\ell)$ from \eqref{ap_indicator}, for all $T\in\cT_\ell$.
  
 (iv){\bf Mark.} Mark a set $\cM_\ell\subset \cT_\ell$ satisfying {\bf Assumption M}.
 
 (v){\bf Refine.} Let $\cT_{\ell+1}:= \mbox{refine}(\cT_\ell,\cM_\ell) $
 }
  \KwOut{Sequences of approximations $(\bsigma_{\ell}^{r}, u_{\ell}^c)$  and corresponding error estimators $\eta_{\ell}^{ls}(\bsigma^r_\ell, u^c_\ell)$.}
\caption{Adaptive Algorithm  with Non-intrusive Least-Squares Functional Estimator}
\end{algorithm}

\subsection{Proof of Plain Convergence} 

By the definition \eqref{ap_indicator} and the triangle inequality, it is easy to see that $\eta^{ls}_\blt(T;\bsigma_\bullet^{r},u_\bullet^c) $ satisfies the local bound,
\begin{equation} 
\eta^{ls}_\blt(T;\bsigma_\bullet^{r},u_\bullet^c) 
    \leq C(\tri (\bsigma_\bullet^{r},u_\bullet^c) \tri  + \| f_1\|_{0,T}+\| \bff_2\|_{0,T}) \quad \forall T\in\cT_\blt.
    \label{local_bdd}
\end{equation}
This is the condition in (2.10) of \cite{Sie:11}.

Given $(\brho,w) \in \bbX$, for all $(\btau,v) \in \bbX$, define the residual $\cR(\brho,w)\in \bbX^*$ by
$$
    \langle \cR(\brho,w),(\btau,v)\rangle =  G_{2s}(\btau,v;\gamma)  -  \cA_{2s}((\brho,w), (\btau,v);\gamma).
$$
\begin{lem} \label{lem_upp}
Given $(\brho,w) \in \bbX$, we have the following upper bound of the residual,
$$
\langle \cR(\brho, w), (\btau,v)\rangle \leq C\sum_{T\in\mathcal{T}_\bullet} \eta^{ls}_\blt (T;\brho,w) \tri (\btau,v)\tri_{T}   \quad \forall (\btau,v) \in \bbX.
$$
\end{lem}
\begin{proof}
By the definition of $\cR$, we have 
\begin{equation} \label{RR}
    \langle \cR(\brho, w), (\btau,v)\rangle 
     = a(w-u,v) + (A^{-1}\bff_2 -A^{-1}\brho-\nabla w,\btau) + (g(\gamma) -B(\brho,w;\gamma)-\gradt \brho,\gradt\btau- (1-\gamma)\bb\cdot A^{-1}\btau).
\end{equation}
It is easy to see that the last two terms on the righthand side of \eqref{RR} are bounded by 
\begin{eqnarray*}
(A^{-1}\bff_2 -A^{-1}\brho-\nabla w,\btau) + (g(\gamma) -B(\brho,w;\gamma)-\gradt \brho,\gradt\btau- (1-\gamma)\bb\cdot A^{-1}\btau)
\leq 
C\sum_{T\in\mathcal{T}_\bullet} \eta^{ls}_\blt (T;\brho, w)\| \btau\|_{H(\divvr;T)}
\end{eqnarray*}
For the term $a(w,v)$, using \eqref{vww} and the facts that $\bsigma= -A\nabla u+\bff_2$ and $\gradt\bsigma  = f_1 - \bb\cdot \nabla u - c u$, we have
\begin{eqnarray*} \nonumber
a(w-u,v) &=&  (\nabla (w-u)+A^{-1}(\brho-\bsigma), A\nabla v+(1-\gamma)\bb v)+(\gradt(\brho-\bsigma)+ B(\brho-\bsigma,w-u;\gamma),v) \\
&=&(A\nabla w + \brho-\bff_2, \nabla v+A^{-1}(1-\gamma)\bb v) +  (\gradt\brho+B(\brho,w;\gamma)-g(\gamma),v)\\
&\leq& 
C\sum_{T\in\mathcal{T}_\bullet} \eta^{ls}_\blt (T;\brho,w) \| v\|_{1,T}.
\end{eqnarray*}
Combining above results, we have the lemma.
\end{proof}
For the solution $(u_\bullet^c, \bsigma_\bullet^{r})$ of the combined problem \eqref{combined_discrete}, we have the following upper bound of the residual from Lemma \ref{lem_upp},
\begin{equation}
   \langle \cR(u_\bullet^c, \bsigma_\bullet^{r}) ,(\btau,v)\rangle \leq C\sum_{T\in\mathcal{T}_\bullet} \eta^{ls}_\blt (T;\bsigma_\bullet^{r},u_\bullet^c ) \tri (\btau,v)\tri_{T}   \quad \forall (\btau,v) \in \bbX.
    \label{upper}
\end{equation}
The result \eqref{upper} is the result in (2.10a) of \cite{Sie:11}.
\begin{rem}
The proof of the above lemma is very similar to the proof of the coercivity of least-squares formulation since essentially it is a middle-step result of the reliability of the least-squares functional error estimator. Like-wisely, we do not use any information that $(\brho, w)$ being the numerical solution of the combined discrete problem in the proof. 
\end{rem}
Then, we are in the position to prove the plain convergence results.
\begin{thm}
Suppose that the marking strategy and the mesh-refinement in Algorithm 1 satisfy {\bf Assumptions (M) and (R1), (R2), (R3)}, then the sequence of approximations $(u_{\ell}^c,\bsigma_{\ell}^r)$ generated by Algorithm 1 satisfies
\begin{equation}
\lim_{\ell\rightarrow \infty}\tri (\bsigma,u) - (\bsigma_\ell^{r},u_\ell^c)\tri = 0 \quad \mbox{and}\quad
            \lim_{\ell\rightarrow \infty} \eta_\ell^{ls}(\bsigma_\ell^{r},u_\ell^{c}) = 0.
\end{equation}
\end{thm}
\begin{proof}
We only need to check the assumptions of Theorem 2.1 of \cite{Sie:11}. We use the same marking strategy and the mesh-refinement strategy as in \cite{Sie:11} and \cite{FP:20}. The norm on the space $\bbX$ is addictive and absolutely continuous, see Proposition \ref{A3A4}, which is the Assumptions (A3) and (A4) of \cite{FP:20} and (2.3) of \cite{Sie:11}. The discrete space satisfies Proposition \ref{S1S2}, which coincides with Assumptions (S1) and (S2) of \cite{FP:20} and  (3.5) of \cite{Sie:11}. It also has the local approximation property \eqref{localapp} on the dense subspace of $\bbX$, which is  the assumption (S3) of \cite{FP:20} and (2.5c) of \cite{Sie:11}. The bilinear form $\cA_{2s}$ is uniform inf-sup stable from Theorem \ref{infsup-cm}, which is the condition (2.6) of  \cite{Sie:11}. Together with \eqref{local_bdd} which is (2.10) of \cite{Sie:11} (or the local boundness condition (L) in \cite{FP:20} for the least-squares FEM). The result \eqref{upper} is the assumption (2.10a) of \cite{Sie:11}. Thus, we can apply Theorem 2.1 of \cite{Sie:11} and show that $\lim_{\ell\rightarrow \infty}\tri (\bsigma,u) - (\bsigma_\ell^{r},u_\ell^c)\tri = 0$. From the reliability and efficiency of the error estimator \eqref{ls_rel_eff_combined}, we have $\lim_{\ell\rightarrow \infty} \eta_\ell^{ls}(\bsigma_\ell^{r},u_\ell^{c}) = 0$.
\end{proof}


\section{A simple monotone nonlinear problem}
\setcounter{equation}{0}
In the following several sections, we extend the idea of non-intrusive least-squares functional error estimator and its plain convergence analysis to a monotone nonlinear problem. We use the same notations introduced in Section 2. We first discuss the model monotone problem and its conforming finite element approximation.  

Consider the following monotone nonlinear elliptic problem,
\begin{equation}
\label{eq_non}
    \left\{\begin{array}{rllll}
        -\Delta u^{nl} + u^{nl}+(u^{nl})^3 &=& f ,\quad &\mbox{in} \ \Omega\\
        u^{nl} &=& 0,  &\mbox{on}\ \partial \Omega.
    \end{array}
    \right.
\end{equation}
Here, for simplicity, we assume that the righthand side $f\in L^2(\O)$. For the case that $f\in H^{-1}(\O)$, it can be treated as the linear case.

For nonlinear problems, the genetic constants may depend on various functions, for example, $f$, the solution $u$, or its approximation. In the rest of the paper, we use the notation $c_f$ to denote a genetic positive constant that may depend on the known right-hand side $f$, but not on the unknown $u$ or its approximation. 

The corresponding weak problem is to find $u^{nl}\in H^1_0(\Omega)$, such that
\begin{equation}
    a_{nl}(u^{nl},v)  = (f,v) \quad \forall v\in H^1_0(\Omega).
    \label{NonlinearCon}
\end{equation}
where 
$$
a_{nl}(w,v): = (\nabla w,\nabla v) + (w,v) +(w^3,v) \quad  \forall w,v\in H^1_0(\Omega).
$$
Note that $a_{nl}(w,v)$ is linear for the second argument.

We can also define the associated nonlinear operator $\mathcal{Z}: w\in H^1_0(\Omega) \rightarrow H^{-1}(\Omega)$ by
$        \langle \mathcal{Z}(w) , v\rangle := a_{nl}(w,v)$, for all $v\in H^1_0(\Omega)$. 
It is easy to see that $\cZ(0)=0$.

By the Sobolev embedding theorem, for a two-dimensional Lipschitz domain $\O$,  $H^1_0(\Omega)\hookrightarrow\hookrightarrow L^p(\Omega)$ $1\leq p<\infty$; for a three-dimensional Lipschitz domain $\O$,  $H^1_0(\Omega)\hookrightarrow L^p(\Omega)$ $1\leq p\leq 6$. Thus, for a $w \in H^1_0(\Omega)$, $w^3 \in L^2(\O)$, and the term $(w^3,v)$ is well-defined for functions $w$ and $v$ in $H^1_0(\Omega)$.

Using the $C^0$-conforming finite element space associated with $\cT_\blt$ defined in \eqref{C0FE}, the corresponding discrete problem is to find $u_\blt^{nl} \in \bbV_\blt$, such that
\begin{equation}
    a_{nl}(u_\blt^{nl},v_\blt) = (f,v_\blt)\quad \forall v_\blt \in \bbV_\blt.
    \label{NonlinearDis}
\end{equation}
\subsection{Some basic analysis on the equation and its finite element approximation}
Due to the fact that 
$(w^3-v^3,w-v)\geq 0$, for all $w,v\in H^1_0(\Omega)$,
the form $a_{nl}(w,v)$ (or $\cZ$) is {strongly monotone} (thus also strictly monotone), that is, 
\begin{equation}
\langle \cZ(w)-\cZ(v), w-v\rangle = a_{nl}(w,w-v) - a_{nl}(v,w-v) \geq  \| w-v\|_1^2 
\quad\forall w,v\in H^1_0(\Omega).
    \label{monotonicity}
\end{equation}
Let $v=0$, we immediately get that 
\beq \label{coercivity_nonlinear}
\langle \cZ (w), w\rangle = a_{nl}(w,w) \geq  \| w\|_1^2  \quad\forall w\in H^1_0(\Omega).
\eeq
Thus, $\cZ$ is coercive in the sense that
$$
\dfrac{\langle \cZ(w),w\rangle}{\|w\|_1} \rightarrow \infty \quad \mbox{as } \|w\|_1 \rightarrow \infty.
$$
For all $z, w, v\in H^1_0(\Omega)$, with Cauchy-Schwarz inequality, generalized H\"{o}lder inequality, and Sobolev embedding theorem, we have
\begin{equation*}
    \begin{aligned}
        |a_{nl}(z,v) - a_{nl}(w,v)| &= |(\nabla z -\nabla w,\nabla v) +(z -w,v) + (z^3 - w^3,v)|\\
        &\leq \|z - w\|_1 \|v\|_1 + |((z^2 + w^2 +zw)(z-w),v)|\\
        &\leq \|z-w\|_1\| v\|_1 + \| z^2+w^2+zw\|_{0,3}\| z-w\|_{0}\| v\|_{0,6}\\
        &\leq \|z-w\|_{1}\| v\|_{1} + C(\|z\|_1^2+\| w\|_1^2)\| z-w\|_1 \| v\|_1\\
        &\leq C(1+\| z\|_1^2+\| w\|_1^2)\| z-w\|_1 \| v\|_{1}.
    \end{aligned}
\end{equation*}
Thus, the mapping $\cZ$ is Lipschitz-continuous for bounded arguments of $z$ and $w$,
\beq \label{Lip-non}
  |\langle \cZ(z)-\cZ(w), v\rangle| =|a_{nl}(z,v) - a_{nl}(w,v)| \leq C(1+\| z\|_1^2+\| w\|_1^2)\| z-w\|_1 \| v\|_{1}
  \quad \forall z, w, v\in H^1_0(\Omega).
\eeq
Let $z=0$, we get
$
|a_{nl}(w,v)| \leq C(1+\| w\|_1^2)\|w\|_1 \| v\|_{1}$, for $w, v\in H^1_0(\Omega)$.

Consider the map 
$$
 t \in \mathbb{R} \rightarrow \langle \cZ(z+tw),v\rangle\in \mathbb{R}\quad \forall z,v,w\in H^1_0(\Omega),
$$ 
It is clear the map is continuous with respect to $t$, thus the operator $\cZ$ is hemicontinuous.

\begin{thm}
The weak problem (\ref{NonlinearCon}) has a unique solution $u^{nl}\in H^1_0(\Omega)$ and the discrete problem (\ref{NonlinearDis}) has a unique solution $u_\blt^{nl}\in \bbV_\blt$. Both solutions have the upper bounds 
\beq \label{nonlinear_stability}
       \| u^{nl} \|_{1} \leq \|f\|_0 \quad \mbox{and}\quad \| u_\blt^{nl} \|_{1} \leq  \|f\|_0. 
\eeq
The following a priori error estimate is true,  
\beq \label{apriori_non}
        \| u^{nl} - u_\blt^{nl} \|_1 \leq c_f \inf_{v_\blt\in \bbV_\blt} \| u^{nl}-v_\blt\|_1.
\eeq
\end{thm}

\begin{proof}
The operator $\cZ$ is hemicontinuous, coercive and strictly monotone on $H^1_0(\O)$. It is also easy to see that $\cZ$ is hemicontinuous, coercive, and strictly monotone on the subspace $\bbV_\blt\subset H^1_0(\O)$. By the Minty-Brower Theorem (Theorem 9.14-1 of \cite{Ciarlet:13} or Theorem 2.K of \cite{Zeidler:AMS109}), both \eqref{NonlinearCon} and \eqref{NonlinearDis} have a unique solution.

The stability results \eqref{nonlinear_stability} are simple consequence of the coercivity \eqref{coercivity_nonlinear}. Note that we can choose the bound to be $\|f\|_{-1}$ when necessary.

It is easy to check that the following error equation is true,
\beq 
\label{error_non}
a_{nl}(u^{nl},v_\blt) - a_{nl}(u_\blt^{nl},v_\blt)  = 0 \quad \forall v_\blt \in \bbV_\blt.
\eeq 
Let $v_\blt$ be an arbitary function in $\bbV_\blt$, by the monotonicity \eqref{monotonicity}, the error equation \eqref{error_non}, the Lipschitz-continuity \eqref{Lip-non}, and the stability \eqref{nonlinear_stability}, we have
\begin{eqnarray*}
\|u^{nl} - u_\blt^{nl}\|_1^2 &\leq& a_{nl}(u^{nl},u^{nl} - u_\blt^{nl}) - a_{nl}(u_\blt^{nl},u^{nl} - u_\blt^{nl}) = a_{nl}(u^{nl},u^{nl} - v_\blt) - a_{nl}(u_\blt^{nl},u^{nl} - v_\blt) \\
&\leq &C(1+ \|u^{nl}\|_1^2 + \|u_\blt^{nl}\|_1^2)\|u^{nl}-u_\blt^{nl}\|_1 \|u^{nl} - v_\blt\|_1 \\
&\leq & C(1+2\|f\|_0^2)\|u^{nl}-u_\blt^{nl}\|_1 \|u^{nl} - v_\blt\|_1.
\end{eqnarray*}
The a priori estimate \eqref{apriori_non} is proved.
\end{proof}

\begin{rem}
For our simple model problem, there are other ways to show that the problems \eqref{NonlinearCon} and  \eqref{NonlinearDis} have unique solutions and to derive a priori error estimates. For example, the problem \eqref{NonlinearCon} can be viewed as the Euler-Lagrangian of the following energy minimization problem:
\beq
\mathtt{E}(u^{nl})  = \inf_{v\in H^1_0(\O)} \mathtt{E}(v)\quad\mbox{where }
\mathtt{E}(v):= \dfrac{1}{2} \|v\|_1^2 + \dfrac{1}{3} \|v^2\|_0^2 - (f,v).
\eeq
The minimization problem has a unique solution based on the convex minimization theory, see \cite{Ciarlet:13}. The a priori error estimate can also be derived from Brezzi-Rappaz-Raviart theory, see \cite{BDMS:15} for details. We choose the monotone operator framework since it is more appropriate for our analysis.
\end{rem}

\section{Least-squares Functional Error Estimator for the model monotone problem}
\setcounter{equation}{0}
Let $\bsigma^{nl} = -\nabla u^{nl}$ in \eqref{eq_non}, then $\gradt \bsigma^{nl} +(u^{nl})^3+u^{nl} =f$. We have the  first-order system:
\beq \label{fosys_non}
\left\{
\begin{array}{rclll}
\bsigma^{nl} + \nabla u^{nl} &=& 0 & \mbox{ in } \O,
\\[1mm]
\gradt\bsigma^{nl} + u^{nl}+(u^{nl})^3  &=&f&  \mbox{ in } \O,\\
u^{nl} &=& 0 &\mbox{ on } \p\O.
\end{array}
\right.
\eeq
For $u^{nl}\in H^1_0(\O)$, we have $(u^{nl})^3\in L^2(\O)$, and $\bsigma^{nl} \in H(\divvr;\O)$. 
For $(\btau,v)\in \bbX$, define the nonlinear least-squares functional for the system \eqref{fosys_non},
\begin{eqnarray} \label{NLS_functional}
\NLS(\btau,v;f) &:=& \|\btau + \nabla v\|_0^2 +\|\gradt\btau +v+v^3-f\|_0^2.
\end{eqnarray}
Given the approximations $(\btau,v)$ of  \eqref{fosys_non}, we can use the nonlinear least-squares functional $\NLS$ as a posteriori error estimator.
\begin{rem}
In contrast to linear problems, the nonlinear least-squares minimization problems and their finite element approximations are less discussed. While the original nonlinear problem \eqref{NonlinearCon} may be convex, its corresponding nonlinear least-squares minimization problem is not necessarily so. Therefore, we can only establish the existence and uniqueness of numerical approximations in the neighborhood of the exact solution with the help of the implicit function theorem. Nevertheless, even the a priori analysis of the original nonlinear least-squares problem is not well-studied, the non-intrusive least-squares functional a posteriori error analysis is still relatively easy. This makes it a useful tool for assessing the accuracy of numerical approximations in nonlinear problems.
\end{rem}

With  the solution of \eqref{NonlinearDis}, $u_\blt^{nl}\in\bbV_\blt$, available, in order to use the least-squares functional error estimator, we need to construct a $\bsigma_\blt\in \bbW_\blt$. By the same principle as the linear problem, we replace the function $v$ in the nonlinear least-squares functional \eqref{NLS_functional} by $u^{nl}$'s approximation $u_\blt^{nl}\in\bbV_\blt\subset H^1_0(\O)$. We get a new functional,
\beq
\NJ(\btau; u_\bullet^{nl}, f) := 
\NLS(\btau,u_\bullet^{nl};f) = \|\btau + \nabla u_\bullet^{nl}\|_0^2 +\|\gradt\btau +(u_\bullet^{nl})^3+u_\bullet^{nl}-f \|_0^2.
\eeq
Then the corresponding minimization problem to find $\bsigma_\bullet^{nr}\in \bbW_\blt$ is:
\beq
\mbox{Find } \bsigma_\bullet^{nr} \in \bbW_\blt  \mbox{ such that }
\NJ(\bsigma_\bullet^{nr};u_\bullet^{nl}, f) = \inf_{\btau\in \bbW_\blt} \NJ(\btau;u_\bullet^{nl}, f). 
\eeq
Or, equivalently: Find $\bsigma_\bullet^{nr} \in \bbW_\blt$, such that,
\beq \label{nl_flux_recovery_discrete}
b_{hdiv}(\bsigma_\bullet^{nr},\btau)
= (-\nabla u_\bullet^{nl},\btau) + (f - u_\bullet^{nl}- (u_\bullet^{nl})^3,\gradt\btau),\quad\forall \btau\in \bbW_\blt,
\eeq
where the $H(\divvr)$-inner product $b_{hdiv}$ is defined as
$$
b_{hdiv}(\brho,\btau):= (\brho,\btau) + (\gradt\brho,\gradt\btau)  \quad \forall \brho,\btau \in H(\divvr;\O).
$$
We define the following continuous problem: Find $\bsigma^{nl}\in H(\divvr;\O)$, such that,  
\beq \label{nl_flux_recovery}
b_{hdiv}(\bsigma^{nl},\btau) = (-\nabla u^{nl},\btau) + (f -u^{nl}- (u^{nl})^3,\gradt\btau)\quad \forall \btau\in H(\divvr;\O).
\eeq
Note that the exact solution $(\bsigma^{nl}, u^{nl})$ of the original first-order system \eqref{fosys_non} satisfies \eqref{nl_flux_recovery}.
We have the following error equation,
\beq \label{erroreq_non}
b_{hdiv}(\bsigma^{nl}-\bsigma_\bullet^{nr},\btau) = -(\nabla u^{nl}-\nabla u_\bullet^{nl},\btau) - (u^{nl}+(u^{nl})^3 - u_\bullet^{nl}- (u_\bullet^{nl})^3,\gradt\btau)\quad \forall \btau\in \bbW_\blt.
\eeq
The problem \eqref{nl_flux_recovery_discrete} can be viewed as a finite element approximation of the continuous problem \eqref{nl_flux_recovery} with the exact solution $u^{nl}$ replaced by its approximation $u_\bullet^{nl}$.

\subsection{A priori estimate of least-squares flux-recovery for the model nonlinear problem}

%

\begin{thm}
For the recovered flux $\bsigma_\blt^{nr}$ of \eqref{nl_flux_recovery_discrete}, we have the 
the following bound:
\begin{equation}
    \|\bsigma_\blt^{nr}\|_{H(\divvr)} \leq c_f,
    \label{SigmaBounded}
\end{equation}
and the following a priori error estimate: 
    \begin{equation}
    \label{apriori_sigma_non}
        \|\bsigma^{nl} - \bsigma_\blt^{nr}\|_{H(\divvr)}\leq  c_f (\inf_{\btau_\blt\in \bbW_\blt} \| \bsigma^{nl}-\btau_\blt\|_{H(\divvr)}+\| u^{nl}-u_\blt^{nl}\|_1).
    \end{equation}
\end{thm}
\begin{proof}
We have the bound from \eqref{nl_flux_recovery_discrete} using the Sobolev embedding theorem, the H\"older inequality, and the stability \eqref{nonlinear_stability}, 
$$
    \|\bsigma_\blt^{nr}\|_{H(\divvr)} \leq C(\| u_\blt^{nl} \|_1 + \| u_\blt^{nl} \|_1^3 + \| f\|_0) \leq c_f.
$$
Using similar arguments as \eqref{Lip-non} and the stability \eqref{nonlinear_stability}, we have
\beq
\| (u^{nl})^3 - (u_\blt^{nl})^3\|_{0} \leq c_f \| u^{nl} - u_\blt^{nl}\|_1.
\label{NonlinearL2Error}
\eeq
Let $\btau_\blt$ be an arbitrary function in $\bbW_\blt$, by the Cauchy-Schwarz inequality, the error equation \eqref{erroreq_non}, \eqref{NonlinearL2Error}, the triangle inequality, and the Young's inequality with $\varepsilon$, we have
\begin{equation*}
\begin{aligned}
\|\bsigma^{nl} - \bsigma_\blt^{nr}\|_{H(\divvr)}^2 & = b_{hdiv}(\bsigma^{nl}-\bsigma_\blt^{nr},\bsigma^{nl}-\bsigma_\blt^{nr}) = b_{hdiv}(\bsigma^{nl}-\bsigma_\blt^{nr},\bsigma^{nl}-\btau_\blt) + b_{hdiv}(\bsigma^{nl}-\bsigma_\blt^{nr},\btau_\blt-\bsigma_\blt^{nr})\\
&\leq \| \bsigma^{nl}-\bsigma_\blt^{nr}\|_{H(\divvr)}\| \bsigma^{nl}-\btau_\blt\|_{H(\divvr)}  \\
& - (\nabla u^{nl}-\nabla u_\blt^{nl},\btau_\blt- \bsigma_\blt^{nr}) - ((u^{nl})^3 +u^{nl}- (u_\blt^{nl})^3-u_\blt^{nl},\gradt (\btau_\blt-\bsigma_\blt^{nr}))\\
&\leq \| \bsigma^{nl}-\bsigma_\blt^{nr}\|_{H(\divvr)}\| \bsigma^{nl}-\btau_\blt\|_{H(\divvr)} + c_f\| u^{nl}-u_\blt^{nl}\|_1 \| \btau_\blt-\bsigma_\blt^{nr}\|_{H(\divvr)}\\
            &\leq \| \bsigma^{nl}-\bsigma_\blt^{nr}\|_{H(\divvr)}\| \bsigma^{nl}-\btau_\blt\|_{H(\divvr)} + c_f\| u^{nl}-u_\blt^{ls}\|_1 (\| \bsigma-\bsigma_\blt^{nr}\|_{H(\divvr)} + \| \bsigma^{nl}-\btau_\blt\|_{H(\divvr)})\\
            &\leq \frac{1}{2}\| \bsigma^{nl} - \bsigma_\blt^{nr}\|_{H(\divvr)}^2 + c_f(\| \bsigma^{nl}-\btau_\blt\|_{H(\divvr)}^2 + \| u^{nl}-u_\blt^{nl}\|_1^2).
\end{aligned}
\end{equation*}
The theorem is then proved.
\end{proof}

\subsection{The least-squares functional error estimator for the model nonlinear problem}

We introduce some notations of the least-squares functional error estimator associated a mesh $\cT_\bullet \in \bbT$. Let  $v_\blt \in \bbV_\blt$  and $\btau_\blt \in \bbW_\blt$ be two arbitrary finite element functions in their spaces associated with the mesh  $\cT_\bullet$, respectively.  For an element $T\in \cT_\blt$, we define the element-wise nonlinear least-squares functional error indicator as
\begin{eqnarray}\label{ap_nl_indicator}
\eta^{nls}_\bullet(T;\btau_\bullet,v_\bullet) &:=& \left(\|\btau_\bullet + \nabla v_\bullet\|_{0,T}^2 +\|\gradt\btau_\bullet + v_\bullet+v_\bullet^3-f \|_{0,T}^2\right)^{1/2}.
\end{eqnarray}
For a collection of elements $\cU_\bullet\subset \cT_\bullet$, we define the least-squares a posteriori error estimator defined on $\cU_\bullet$,
\beq \label{ap_nl_indicator_cU}
\eta^{nls}_\bullet(\cU_\bullet; \btau_\bullet,v_\bullet):= \sqrt{ \sum_{T\in\cU_\blt}
\eta^{nls}_\bullet(T;\btau_\bullet,v_\bullet)^2}.
\eeq
For the case $\cU_\blt = \cT_\blt$, we use a simpler notation,
\beq \label{ap_nl_estimator}
\eta^{nls}_\bullet(\btau_\bullet,v_\bullet) := \eta^{nls}_\bullet(\cT_\bullet; \btau_\bullet,v_\bullet).
\eeq
We also have
\beq \label{nlsap}
\eta^{nls}_\bullet(\btau_\bullet,v_\bullet) = \NLS(\btau_\bullet,v_\bullet;f)^{1/2} 
\eeq
Let $\cT_\bullet \in \bbT$. We have computed $u_\bullet^{nl} \in \bbV_\blt$ from the conforming finite element discrete problem \eqref{NonlinearDis} and recovered a numerical flux $\bsigma_\bullet^{nr} \in \bbW_\blt$ from \eqref{nl_flux_recovery_discrete}.  The non-intrusive least-squares functional a posteriori error estimator for nonlinear problem is: 
$$\eta^{nls}_\blt(\bsigma_\bullet^{nr},u_\bullet^{nl}).
$$
Same as the linear case, we view the nonlinear discrete problem \eqref{NonlinearDis} and the flux-recovery problem \eqref{nl_flux_recovery} as a combined two-step problem. 
Similar to the linear case, we show the error equivalence for an arbitary $(\btau,v)\in \bbX$ first, then we show 
the error estimator $\eta^{nls}_\blt(\bsigma_\bullet^{nr},u_\bullet^{nl})$ is then efficient and reliable.
\begin{thm}
For an arbitrary $(\btau,v)\in \bbX$, the error estimator $\eta^{nls}_\blt(\btau,v)$ is reliable and locally efficient: 
\begin{equation} \label{rel_nl1}
\|u^{nl} - v\|_1 + \| \bsigma^{nl}-\btau\|_{H(\divvr)} \leq C(c_f+\|v\|_1^2) \eta^{nls}_\blt(\btau,v),
\end{equation}
    and
\begin{equation} 
\label{nl_eff1}
       \eta^{nls}_\blt(T;\btau,v)\leq C(c_f+\|v\|_1) (\|u^{nl}-v\|_{1,T}+\| \bsigma^{nl}-\btau\|_{H(\divvr;T)}), \quad \forall T\in \mathcal{T}_\blt. 
\end{equation}
In addition, the following reliability bound is independent of the right-hand side $f$,
\beq
\label{rel_nl21}
\|u^{nl} - v\|_1 + \|\bsigma^{nl}-\btau\|_{0} \leq 3\eta^{nls}_\blt(\btau,v).
\eeq
\end{thm}
\begin{proof}
For any $(\btau,v)\in \bbX$ and $w\in H^1_0(\O)$, we have
\beq \label{change}
(\nabla u^{nl}-\nabla v,\nabla w) = (\nabla u^{nl} + \btau,\nabla w)-(\btau+\nabla v,\nabla w) = -(\Delta u^{nl} +\gradt \btau,w)-(\btau+\nabla v,\nabla w). 
\eeq
By the monotonicity \eqref{monotonicity}, \eqref{change} with $w=u^{nl}-v$, and the fact that $-\Delta u^{nl} + (u^{nl})^3 + u^{nl} =f$, we get 
\begin{equation*}
\begin{aligned}
\|u^{nl} - v\|_1^2 &\leq a_{nl}(u^{nl}, u^{nl} - v)-a_{nl}(v,u^{nl} - v) =(\nabla u^{nl}-\nabla v,\nabla (u^{nl}- v)) + ((u^{nl})^3- v^3+u^{nl}-v,u^{nl}-v)\\
&= -(\Delta u^{nl} +\gradt \btau,u^{nl}- v)-(\btau+\nabla v,\nabla (u^{nl}- v)) + ((u^{nl})^3- v^3+u^{nl}-v, u^{nl}-v)
\\
&= -(\btau+\nabla v,\nabla (u^{nl}-v)) + (f-\nabla \cdot \btau - v^3 -v,u^{nl}-v)
    \leq \eta^{nls}_\blt(\btau,v)\| u^{nl}-v\|_1.
\end{aligned}
\end{equation*}
Thus, we have 
\beq \label{es1}
\|u^{nl} - v\|_1  \leq \eta^{nls}_\blt(\btau,v) \quad \forall (\btau,v)\in \bbX.
\eeq
By the fact $\bsigma^{nl} = -\nabla u^{nl}$, the triangle inequality, and \eqref{es1}, we have
$$
\| \bsigma^{nl} - \btau\|_0 = \|\nabla u^{nl} + \btau\|_0 \leq \|\btau + \nabla v\|_0 + \|\nabla(u^{nl} - v)\|_0  \leq 2\eta^{nls}_\blt(\btau,v) \quad \forall (\btau,v)\in \bbX.
$$     
The result \eqref{rel_nl21} is proved.

By the fact $f= \gradt \bsigma^{nl}+ u^{nl}+(u^{nl})^3$,  the triangle inequality, and (\ref{NonlinearL2Error}), we have
\begin{eqnarray*}
\|\nabla \cdot \bsigma^{nl} - \nabla \cdot \btau\|_0 
&\leq& \| \nabla \cdot \btau + v^3 +v-f\|_0 + \| (u^{nl})^3 -v^3\|_0 +\| u^{nl}-v\|_0  \\
&\leq& \| \nabla \cdot \btau + v^3 +v-f\|_0+ C(c_f+\|v\|_1^2)\| u^{nl}-v\|_1 
\\
& \leq&  C(c_f+\|v\|_1^2) \eta^{nls}_\blt(\btau,v).
\end{eqnarray*}
%
%
Combined the above results, we have \eqref{rel_nl1}.

By the facts  $\bsigma^{nl} = -\nabla u^{nl}$ and  $f= \gradt \bsigma^{nl}+ u^{nl}+(u^{nl})^3$, and the triangle inequality, we have
\begin{eqnarray*}
\eta^{nls}_\blt(T;\btau,v) &\leq & \|\btau+ \nabla v\|_{0,T} + \| \nabla \cdot \btau + v^3 + v_\blt -f \|_{0,T}\\
&\leq & \| \bsigma^{nl} - \btau\|_{0,T} + \| \nabla u^{nl} - \nabla v\|_{0,T} 
+ \| \nabla \cdot (\bsigma^{nl} - \btau)\|_{0,T} + \| (u^{nl})^3 - v^3\|_{0,T}    + \| u^{nl}- v\|_{0,T}\\
&\leq & C(c_f+\|v_\blt\|_1)  (\| u^{nl} - v\|_{1,T} + \| \bsigma^{nl}-\btau\|_{H(\divvr;T)}).
    \end{eqnarray*}
The efficiency \eqref{nl_eff1} is proved.
\end{proof}

Choosing $(\btau,v)$ to be $(\bsigma_\bullet^{nr},u_\bullet^{nl})$ in the above theorem and using the stability \eqref{nonlinear_stability}, we have the following result.
\begin{eqnarray} \label{rel_nl2}
\|u^{nl} - u_\blt^{nl}\|_1 + \|\bsigma^{nl}-\bsigma_\blt^{nr}\|_{0} 
&\leq& 3\eta^{nls}_\blt(\bsigma_\bullet^{nr},u_\bullet^{nl}),
\\ \label{rel_nl}
\|u^{nl} - u_\blt^{nl}\|_1 + \| \bsigma^{nl}-\bsigma_\blt^{nr}\|_{H(\divvr)} &\leq& c_f\eta^{nls}_\blt(\bsigma_\bullet^{nr},u_\bullet^{nl}),
\end{eqnarray}
    and
\begin{equation} 
\label{nl_eff}
       \eta^{nls}_\blt(T;\bsigma_\bullet^{nr},u_\bullet^{nl})\leq c_f(\|u^{nl}-u_\blt^{nl}\|_{1,T}+\| \bsigma-\bsigma_\blt^{nr}\|_{H(\divvr;T)}), \quad \forall T\in \mathcal{T}_\blt. 
\end{equation}

The following lemmas are useful for the plain convergence.
\begin{lem}
For any $T\in \cT_\blt$,  the local indicator $\eta^{nls}_\blt(T;\bsigma_\bullet^{nr},u_\bullet^{nl})$ is stable in the following sense:
\begin{equation}
\eta^{nls}_\blt(T;\bsigma_\bullet^{nr},u_\bullet^{nl}) \leq c_f(\| u_\blt^{nl}\|_{1,T}+ \| \bsigma_\blt^{nr}\|_{H(\divvr;T)} + \| f\|_{0,T}).
        \label{NonlinearStable}
\end{equation}
The estimator $\eta^{nls}_\blt(\bsigma_\bullet^{nr},u_\bullet^{nl})$ is uniformly bounded with respect to $\cT_\blt$ in the following sense,
\begin{equation} \label{bdd_eta_nl}
        \eta^{nls}_\blt(\bsigma_\bullet^{nr},u_\bullet^{nl})\leq c_f.
\end{equation}
\end{lem}
\begin{proof}
By the triangle inequality and \eqref{nonlinear_stability}, we have
\begin{equation*}
\begin{aligned}
\eta^{nls}_\blt(T;\bsigma_\bullet^{nr},u_\bullet^{nl}) & \leq C(\| \bsigma_\blt^{nr}\|_{0,T} + \| \nabla u_\blt^{nl}\|_{0,T} + \| \nabla \cdot \bsigma_\blt^{nr}\|_{0,T} + \|u_\blt^{nl}\|_{0,T}+ \| (u_\blt^{nl})^3\|_{0,T} + \| f\|_{0,T})\\
            &\leq C(\| \bsigma_\blt^{nr}\|_{H(\divvr;T)} + c_f \|u_\blt^{nl}\|_{1,T} + \| f\|_{0,T})
            \leq c_f (\| u_\blt^{nl}\|_{1,T}+ \| \bsigma_\blt^{nr}\|_{H(\divvr;T)} + \| f\|_{0,T}).
        \end{aligned}
\end{equation*}
Then by \eqref{nonlinear_stability} and \eqref{SigmaBounded}, we have \eqref{bdd_eta_nl},
\end{proof}

%
%

Given $(\brho,w) \in \bbX$, for all $(\btau,v) \in \bbX$, define the residual $\cR^{nl}(\brho,w)\in \bbX^*$ by
\beq
\langle\cR^{nl}(\brho,w), (\btau,v) \rangle := a_{nl}(w,v) -(f,v)  +(\brho+\nabla w,\btau) + (\gradt\brho+w^3+w-f,\nabla \cdot \btau).
\eeq

\begin{lem}
Given $(\brho,w) \in \bbX$, we have the following upper bound of the residual,
\beq
\langle\cR^{nl}(\brho,w), (\btau,v) \rangle \leq \sum_{T\in \cT_\blt }\eta^{nls}_\blt(T;\brho,w) \tri (\btau,v)\tri_T \quad \forall (\btau,v) \in \bbX.
\eeq

\end{lem}
\begin{proof}
By the definition of $\cR^{nl}$, integration by parts, and Cauchy-Schwarz inequality, we have
\begin{eqnarray*}
&&\langle\cR^{nl}(\brho,w), (\btau,v) \rangle =  a_{nl}(w,v) - (f,v) + (\brho+w,\btau) + (\nabla \cdot \brho +w^3 + w - f,\nabla \cdot \btau) \\
&&=  (\nabla w,\nabla v) + (w^3 +w -f,v)+(\brho+\nabla w,\btau) + (\nabla \cdot \brho + w^3+w - f,\nabla \cdot \btau)  \\
 && =  (\nabla w + \brho,\nabla v+\btau) - (\brho, \nabla v) + (w^3 +w -f,v) + (\nabla \cdot \brho +w^3 +w - f,\nabla \cdot \btau) \\
&& = (\nabla w + \brho ,\nabla v+\btau) + (\nabla\cdot \brho + w^3 +w - f,v+\nabla \cdot \btau) \\
            &&\leq \sum_{T\in \cT_\blt} (\|\nabla w + \brho \|_{0,T}(\|\nabla v\|_{0,T}+\|\btau\|_{0,T}) + \|\nabla\cdot \brho +w^3 +w - f\|_{0,T}(\| v\|_{0,T}+\|\gradt\btau\|_{0,T})\\
            &&\leq \sum_{T\in\cT_\blt}\eta^{nls}_\blt(T;\brho,w)\tri (\btau,v)\tri_T \quad \forall (\btau,v) \in \bbX.
\end{eqnarray*}
The lemma is proved.
\end{proof}
Choosing $(\brho,w)$ to be $(\bsigma_\blt^{nr}, u_\bullet^{nl})$ in the above lemma, we have
\beq
\langle\cR^{nl}(\bsigma_\blt^{nr},u_\bullet^{nl}), (\btau,v) \rangle \leq \sum_{T\in \cT_\blt }\eta^{nls}_\blt(T;\bsigma_\bullet^{nr},u_\bullet^{nl}) \tri (\btau,v)\tri_T \quad \forall (\btau,v) \in \bbX.
        \label{NonlinearResidualBound}
\eeq

\section{Plain Convergence for the Adaptive Algorithm for the model nonlinear case}
\setcounter{equation}{0}

In this section, we prove the plain convergence for the adaptive algorithm driven by the non-intrusive least-squares functional estimator for the model nonlinear problem. We use the following adaptive algorithm for the nonlinear case. It is almost identical to Algorithm 1 with some necessary changes.

\begin{algorithm}[H]
 \KwIn{Initial triangulation $\cT_0$.}

 \For{$\ell = 0,1,2,\cdots$ }{
  (i) {\bf Solve}. Compute the discrete solution $u_{\ell}^{nl}\in \bbV_\ell$ by solving \eqref{NonlinearDis}.
  
  (ii){\bf Nonlinear Least-Squares Recovery}. Recover $\bsigma_{\ell}^{nr}\in \bbW_\ell$ by solving the least-squares recovery problem \eqref{nl_flux_recovery_discrete}. 
    
  (iii){\bf Estimate.} Compute $\eta_{\ell}^{nls}(T;\bsigma^{nr}_\ell,u^{nl}_\ell)$ from \eqref{ap_nl_indicator}, for all $T\in\cT_\ell$.
  
 (iv){\bf Mark.} Mark a set $\cM_\ell\subset \cT_\ell$ satisfying {\bf Assumption M} using $\eta_{\ell}^{nls}(T;\bsigma^{nr}_\ell,u^{nl}_\ell)$.
 
 (v){\bf Refine.} Let $\cT_{\ell+1}:= \mbox{refine}(\cT_\ell,\cM_\ell) $
 }
  \KwOut{Sequences of approximations $(\bsigma_{\ell}^{nr}, u_{\ell}^{nl})$  and corresponding error estimators $\eta_{\ell}^{nls}(\bsigma^{nr}_\ell, u^{nl}_\ell)$.}
\caption{Adaptive Algorithm  with Non-intrusive Least-Squares Functional Estimator for Nonlinear Problem}
\end{algorithm}

Define the following spaces
$$
\bbV_{\infty}:= \overline{\bigcup_{\ell\geq 0}\bbV_\ell } \subset H^1_0(\Omega) \quad\mbox{and}\quad
\bbW_{\infty}:= \overline{\bigcup_{\ell\geq 0}\bbW_\ell }\subset H(\divvr;\Omega).
$$
Consider the following two problems in $\bbV_{\infty}$ and $\bbW_{\infty}$.

Find $u_{\infty}^{nl} \in \bbV_{\infty}$ such that
\beq \label{uinf}
a_{nl}(u_{\infty}^{nl},v_{\infty}) = (f,v_{\infty})\quad \forall v_{\infty}\in \bbV_{\infty},
\eeq
and find $\bsigma_{\infty}^{nr}\in \bbW_\infty$ such that
\beq \label{sigmainf}
b_{hdiv}(\bsigma_{\infty}^{nr},\btau_{\infty}) = -(\nabla u_{\infty}^{nl},\btau_{\infty}) + (f-(u_{\infty}^{nl})^3-u_{\infty}^{nl},\nabla \cdot \btau_{\infty})      \quad\forall \btau_{\infty}\in \bbW_{\infty}.
\eeq

We have the following lemma.
\begin{lem}
Problem (\ref{uinf}) has a unique solution $u^{nl}_\infty\in \bbV_\infty$ and problem (\ref{sigmainf}) has a unique solution $\bsigma_{\infty}^{nr}\in \bbW_\infty$. 
Both solutions have the upper bounds, 
\beq \label{nonlinear_stability_inf}
       \| u^{nl}_\infty \|_{1} \leq \|f\|_0 \quad \mbox{and}\quad \| \bsigma_{\infty}^{nr} \|_{H(\divvr;\O)} \leq  c_f. 
\eeq
The solution $u_{\infty}^{nl}$ and $\bsigma_{\infty}^{nr}$ satisfy the following convergence results,
\begin{eqnarray}
        \| u_{\infty}^{nl} - u_{\ell}^{nl}\|_1 \rightarrow 0 \quad \mbox{and} \quad \| \bsigma_{\infty}^{nr} - \bsigma_{\ell}^{nr}\|_{H(\divvr;\O)} \rightarrow 0
         \quad \mbox{as}\quad \ell  \rightarrow  \infty.
        \label{UkConvergenceNonlinear} 
\end{eqnarray}
\end{lem}
\begin{proof}
Due to the fact the operator $\cT$ is hemocontinuous, coercive, and strictly monotone on $\bbV_{\infty}$, we have the existence and uniqueness of \eqref{uinf}. The stability that $\| u^{nl}_\infty \|_{1} \leq \|f\|_0$ is the consequence of the coercivity \eqref{coercivity_nonlinear}.

By the fact $\bbV_\ell \subset \bbV_\infty$, the solution $u_\ell^{nl} \in \bbV_\ell$ can be viewed as the Galerkin projection of $u^{nl}_\infty$ in  $\bbV_\ell$. Thus, we have the a priori error estimates similar to \eqref{apriori_non},
\beq 
\| u^{nl}_\infty - u_\ell^{nl} \|_1 \leq c_f \inf_{v_\ell\in \bbV_\ell} \| u^{nl}_\infty-v_\ell\|_1.
\eeq
By the definition of $\bbV_\infty$, the space $\{\bbV_\ell\}_{\ell \geq 0}$ is dense in $\bbV_{\infty}$, thus we have $\| u_{\infty}^{nl} - u_{\ell}^{nl}\|_1 \rightarrow 0$.

The existence and uniqueness of \eqref{sigmainf} is obvious due to the definition of $b_{hdiv}$. The stability  $ \| \bsigma_{\infty}^{nr} \|_{H(\divvr;\O)} \leq  c_f$ can be proved in a similar fashion as that of \eqref{SigmaBounded}. By the fact $\bbW_\ell \subset \bbW_\infty$, we can get the following a priori estimate similar to \eqref{apriori_sigma_non},
\begin{equation}
\| \bsigma_{\infty}^{nr}  - \bsigma_\ell^{nr}\|_{H(\divvr)}\leq  c_f (\inf_{\btau\in \bbW_\ell} \| \bsigma_{\infty}^{nr} -\btau\|_{H(\divvr)}+\|u^{nl}_\infty -u_\ell^{nl}\|_1).
\end{equation}
The space $\{\bbW_\ell\}_{\ell \geq 0}$ is dense in $\bbW_{\infty}$, thus we have $\| \bsigma_{\infty}^{nr}  - \bsigma_\ell^{nr}\|_{H(\divvr)} \rightarrow 0$.
\end{proof}
%
%
%

\begin{lem} ({\bf convergence of estimator on marked elements})
Suppose that the marking strategy and the mesh-refinement in Algorithm 2 satisfy {\bf Assumptions (M) and (R1), (R2), (R3)}, then 
\begin{equation}
\lim_{\ell \rightarrow \infty} \max\{\eta^{nls}_\ell(T;\bsigma_\ell^{nr},u_\ell^{nl}) : T\in \mathcal{M}_\ell \} = 0.
        \label{EtaMarkConvergenceNonlinear}
\end{equation}
\end{lem}
\begin{proof}
By the triangle inequality, we have the following property,
\begin{eqnarray*}
\eta^{nls}_\ell(T;\bsigma_\ell^{nr},u_\ell^{nl})^2 &=& \|\nabla u_\ell^{nl} + \bsigma_\ell^{nr} \|_{0,T}^2  + \|\nabla\cdot \bsigma_\ell^{nr} +(u_\ell^{nl})^3 +u_\ell^{nl} - f\|_{0,T}^2  \\
&\leq &
\|\bsigma_\infty^{nr} + \nabla u_\infty^{nl} \|_{0,T}^2 + \|\nabla\cdot \bsigma_\infty^{nr} +(u_\infty^{nl})^3 +u_\infty^{nl} - f\|_{0,T}^2 
 +\|\nabla u_\ell^{nl} - \nabla u_\infty^{nl} \|_{0,T}^2 \\ 
 && +  \|\bsigma_\ell^{nr}  - \bsigma_\infty^{nr} \|_{0,T}^2   
 + \|\nabla\cdot \bsigma_\ell^{nr} - \gradt\bsigma_\infty^{nr} \|_{0,T}^2 + \|(u_\ell^{nl})^3 +u_\ell^{nl} - (u_\infty^{nl})^3 -u_\infty^{nl} \|_{0,T}^2. 
\end{eqnarray*}
When $\ell \rightarrow \infty$, the terms $\|\nabla u_\ell^{nl} - \nabla u_\infty^{nl} \|_{0,T}$, $\|\bsigma_\ell^{nr}  - \bsigma_\infty^{nr} \|_{0,T}$, and $\|\nabla\cdot \bsigma_\ell^{nr} - \gradt\bsigma_\infty^{nr} \|_{0,T}$ converges to zero due to \eqref{UkConvergenceNonlinear}. The term $\|(u_\ell^{nl})^3 +u_\ell^{nl} - (u_\infty^{nl})^3 -u_\infty^{nl} \|_{0,T} \leq c_f\|u_\ell^{nl}-u_\infty^{nl}\|_{1,T}$ also tends to zero.

For the element $T_\ell =  \arg\max_{T\in \cM_\ell} \eta^{nls}_\blt(T;\bsigma_\ell^{nr},u_\ell^{nl})$, its mesh size tends to zero as argued in the proof of Lemma 3.6 of \cite{Sie:11}, combined with  \eqref{nonlinear_stability_inf}, we have $\|\bsigma_\infty^{nr} + \nabla u_\infty^{nl} \|_{0,T_\ell}^2 + \|\nabla\cdot \bsigma_\infty^{nr} +(u_\infty^{nl})^3 +u_\infty^{nl} - f\|_{0,T_\ell}^2 \rightarrow 0$ as $\ell\rightarrow 0$. Thus we have the lemma.
\end{proof}

\begin{lem}
    \label{lem_weakconv}
    Suppose that the marking strategy and the mesh-refinement in Algorithm 2 satisfy {\bf Assumptions (M) and (R1), (R2), (R3)}, then we have the following weak convergence of the residual,
\begin{eqnarray}
\label{res_conv}
\lim_{\ell\rightarrow \infty} \langle\cR^{nl}(\bsigma_\ell^{nr},u_\ell^{nl}), (\btau,v) \rangle  &=& 0 \quad \forall (\btau,v) \in H^2(\O)^d \times H^2(\O)\cap H^1_0(\O) .
\end{eqnarray}
\end{lem}
\begin{proof}
By the definition of $\cR^{nl}$, \eqref{error_non}, and \eqref{nl_flux_recovery_discrete}, we have the following indentity,
\begin{eqnarray*}
\langle\cR^{nl}(\bsigma_\ell^{nr},u_\ell^{nl}), (\btau_\ell,v_\ell) \rangle  &= & a_{nl}(u_\ell^{nl},v_\ell) - a_{nl}(u,v_\ell) +b_{hdiv}(\bsigma_\ell^{nr},\btau_\ell) +(\nabla u_\ell^{nl},\btau_\ell) + ((u_\ell^{nl})^3+u_\ell^{nl}-f,\nabla \cdot \btau_\ell) \\
&=& 0 \quad\quad\forall v_\ell\in \bbV_\ell, \btau_\ell\in \bbW_\ell,
\end{eqnarray*}
By \eqref{NonlinearResidualBound},  we have
\begin{eqnarray*}
\langle\cR^{nl}(\bsigma_\ell^{nr},u_\ell^{nl}), (\btau,v) \rangle  &=& \langle\cR^{nl}(\bsigma_\ell^{nr},u_\ell^{nl}), (\btau-\btau_\ell,v-v_\ell) \rangle  \\ 
&\leq& \sum_{T\in \cT_\ell }\eta^{nls}_\ell(T;\bsigma_\ell^{nr},u_\ell^{nl}) (\|v-I_{\ell}v\|_{1,T}+\|\btau-I_{\ell}^{rt}\btau_\ell\|_{H(\divvr;T)}).
\end{eqnarray*}
The sum of the right-hand side can be handled as in the Proposition (3.7) of \cite{Sie:11}, and we then have the result of the lemma.
\end{proof}
 
\begin{thm}
Suppose that the marking strategy and the mesh-refinement in Algorithm 2 satisfy {\bf Assumptions (M) and (R1), (R2), (R3)}, then the sequence of approximations $(u_{\ell}^{nl},\bsigma_{\ell}^{nr})$ generated by Algorithm 2 satisfies
\begin{equation}
       \lim_{\ell \rightarrow \infty } ( \| u - u_\ell^{nl}\|_1 + \| \bsigma - \bsigma_\ell^{nr}\|_{H(\divvr)} + 
        \eta^{nls}_\ell(\bsigma_\ell^{nr},u_\ell^{nl})  ) = 0.
\end{equation}
\end{thm}
\begin{proof}
By the definition of $\cR^{nl}$, the Lipschitz-continuity of $a_{nl}$ \eqref{Lip-non}, the stabilities \eqref{nonlinear_stability} and \eqref{nonlinear_stability_inf}, and \eqref{UkConvergenceNonlinear}, for all $(\btau,v) \in  H^2(\O)^d \times H^2(\O)\cap H^1_0(\O) $, we have
\begin{eqnarray*}
\langle\cR^{nl}(\bsigma_\infty^{nr},u_\infty^{nl}), (\btau,v) \rangle &=  & a_{nl}(u_\infty^{nl},v) - a_{nl}(u_{\ell}^{nl},v) 
+b_{hdiv}(\bsigma_{\infty}^{nr} - \bsigma_\ell^{nr},\btau) + ((\nabla u_{\infty}^{nl} - \nabla u_\ell^{nl}),\btau)  \\
&& + ((u_{\infty}^{nl})^3+u_{\infty}^{nl}-(u_\ell^{nl})^3-u_\ell^{nl},\nabla\cdot \btau)
+ \langle\cR^{nl}(\bsigma_\ell^{nr},u_\ell^{nl}), (\btau,v) \rangle\\
 &\leq & C(1+\| u_{\infty}^{nl}\|_1^2+\| u_\ell^{nl}\|_1^2)\| u_{\infty}^{nl} -u_\ell^{nl}\|_1\|v\|_1 + 
 \| \bsigma_{\infty}^{nr} - \bsigma_\ell^{nr} \|_{H(\divvr)}\|\btau\|_{H(\divvr)} \\
 && + C(1+\| u_{\infty}^{nl}\|_1^2+\| u_\ell^{nl}\|_1^2)  \| u_{\infty}^{nl} - u_\ell^{nl} \|_1\|\btau\|_{H(\divvr)}+ \langle\cR^{nl}(\bsigma_\ell^{nr},u_\ell^{nl}), (\btau,v) \rangle\\
  &\leq &c_f (\| u_{\infty}^{nl} -u_\ell^{nl}\|_1+ \| \bsigma_{\infty}^{nr} - \bsigma_\ell^{nr} \|_{H(\divvr)})
  \tri (\btau,v)\tri +  \langle\cR^{nl}(\bsigma_\ell^{nr},u_\ell^{nl}), (\btau,v) \rangle \rightarrow 0.
\end{eqnarray*}
The space $H^2(\O)^d \times H^2(\O)\cap H^1_0(\O)$ is dense in $\bbX$, thus we have 
\beq
\langle\cR^{nl}(\bsigma_\infty^{nr},u_\infty^{nl}), (\btau,v) \rangle  = 0 \quad \forall (\btau,v) \in \bbX.
\eeq
By \eqref{monotonicity}, we have
\begin{equation*}
\|u^{nl} - u_{\infty}^{nl}\|_1^2 \leq a_{nl}(u_{\infty}^{nl},u_{\infty}^{nl}-u^{nl}) - a_{nl}(u^{nl},u_{\infty}^{nl}-u^{nl}) = \langle\cR^{nl}(\bsigma_\infty^{nr}, u_\infty^{nl}), (0, u_{\infty}^{nl}-u^{nl}) \rangle = 0.
\end{equation*}
Thus we have $u^{nl} = u_{\infty}^{nl}$ and   $\lim_{\ell \rightarrow \infty }\| u^{nl} - u_\ell^{nl}\|_1 =0$.
 
With $u^{nl} = u_{\infty}^{nl}$, we then have 
$\nabla u_{\infty}^{nl} =\nabla u^{nl} = -\bsigma^{nl}$ and $(u^{nl})^3+u^{nl}-f = -\gradt\bsigma^{nl}$.
Thus
\begin{eqnarray*}
0&=&\langle\cR^{nl}(\bsigma_\infty^{nr},u_\infty^{nl}), (\bsigma^{nl} - \bsigma_{\infty}^{nr}, 0) \rangle \\
&=&
b_{hdiv}(\bsigma_\infty^{nr},\bsigma^{nl} - \bsigma_{\infty}^{nr}) +(\nabla u^{nl},\bsigma^{nl} - \bsigma_{\infty}^{nr}) + ((u^{nl})^3+u^{nl}-f,\nabla \cdot (\bsigma^{nl} - \bsigma_{\infty}^{nr})) \\
&=&
b_{hdiv}(\bsigma_\infty^{nr},\bsigma^{nl} - \bsigma_{\infty}^{nr}) -(\bsigma^{nl},\bsigma^{nl} - \bsigma_{\infty}^{nr}) - (\gradt\bsigma^{nl},\nabla \cdot (\bsigma^{nl} - \bsigma_{\infty}^{nr})) \\
&=& b_{hdiv}(\bsigma^{nl} - \bsigma_{\infty}^{nr},\bsigma^{nl} - \bsigma_{\infty}^{nr}) = \|\bsigma^{nl} - \bsigma_{\infty}^{nr}\|_{H(\divvr)}^2.
\end{eqnarray*}
We have $\bsigma^{nl} = \bsigma_{\infty}^{nr}$ and   $\lim_{\ell \rightarrow \infty }\|\bsigma^{nl} - \bsigma_{\ell}^{nr}\|_{H(\divvr)} =0$.

By \eqref{nl_eff}, we have $\lim_{\ell \rightarrow \infty } \eta^{nls}_\ell(\bsigma_\ell^{nr},u_\ell^{nl}) = 0$. The proof is completed.
\end{proof}

\section{Concluding Remarks}
In this paper, we present a systematic approach for applying the least-squares functional error estimator to problems that are not solved by the LSFEM. We accomplish this by recovering a physically meaningful auxiliary variable through the minimization of the corresponding least-squares functional, where part of the solution is substituted with the available approximation. By treating the solving and recovery process as a combined two-step problem, we rigorously establish a priori and non-intrusive least-squares functional a posteriori error estimates.
Plain convergences are proved for general second-order elliptic equation and a model monotone problem with adaptive algorithms driven by the non-intrusive least-squares functional estimators.

We only consider a simple model non-linear problem in this paper.  In our future research, we intend to apply the methodology to more complicated problems such as Navier-Stokes equations. 
In \cite{CZ:10b}, numerical experiments were conducted for elliptic problems. Extensive numerical experiments will be provided in forth-coming papers for a wide range of linear and nonlinear problems.

\bibliographystyle{plain}

\end{document}